\documentclass[preprint, 12pt]{article}

\pdfoutput=1

\usepackage{amsmath,amsthm,amsfonts,amssymb,amscd}
\usepackage{graphicx,bm,color}
\usepackage{algorithm}
\usepackage[algo2e]{algorithm2e}
\usepackage{amsfonts}
\usepackage{algpseudocode}
\usepackage{stmaryrd}
\usepackage{listings}
\usepackage{lineno}
\usepackage{mathptmx}
\usepackage{cleveref}
\usepackage{tikz}
\usepackage[skins,theorems]{tcolorbox}
\usepackage{authblk}
\usepackage{tkz-graph}
\usepackage[latin1]{inputenc}
\usetikzlibrary{shapes,arrows}
\usepackage{smartdiagram}
\usetikzlibrary{arrows,calc, decorations.markings, intersections}
\usepackage{caption}
\usepackage{cancel}
\usepackage{mathtools}
\usepackage{url}
\usepackage{float}
 
\definecolor{gray}{gray}{0.6}

\theoremstyle{plain}
\newtheorem{thm}{Theorem} 
\newtheorem{lem}[thm]{Lemma}

\theoremstyle{remark}
\newtheorem{remark}{Remark}
\newtheorem*{remark*}{Remark}

\newcommand{\vertiii}[1]{{\left\vert\kern-0.25ex\left\vert\kern-0.25ex\left\vert #1 
		\right\vert\kern-0.25ex\right\vert\kern-0.25ex\right\vert}}

\newcommand\norm[1]{\left\lVert#1\right\rVert}

\begin{document}

\title{Error indicator for the incompressible Darcy flow problems using Enhanced Velocity \\ Mixed Finite Element Method.}


\author[1,2]{Yerlan Amanbek}
\author[1]{Gurpreet Singh}
\author[1]{Gergina Pencheva}
\author[1]{Mary F. Wheeler}
\affil[1]{Oden Institute for Computational Engineering and Sciences, University of Texas at Austin}
\affil[2]{Nazarbayev University}
\affil[ ]{\textit{yerlan.amanbek@nu.edu.kz, \{gurpreet, gergina \}@utexas.edu, mfw@ices.utexas.edu}}


\date{\today}

\maketitle              
\begin{abstract}
	In the flow and transport numerical simulation, mesh adaptivity strategy is important in reducing the usage of CPU time and memory. The refinement based on the pressure error estimator is commonly-used approach without considering the flux error which plays important role in coupling flow and transport systems. 
	We derive a \textit{posteriori} error estimators for Enhanced Velocity Mixed Finite Element Method (EVMFEM) in the incompressible Darcy flow. 
	We show numerically difference of the explicit residual based error estimator and implicit error estimators, where Arbogast and Chen post-processing procedure from \cite{arbogast1995implementation} for pressure was used to improve estimators. 
	A residual-based error estimator provides a better indicator for pressure error. Proposed estimators are good indicators in finding of the large error element. Numerical tests confirm theoretical results.
	We show the advantage of pressure postprocessing on the detecting of velocity error. To the authors' best knowledge, a posteriori error analysis of EVMFEM has been scarcely investigated from the theoretical and numerical point of view. 
\end{abstract}

\section{Introduction}
In subsurface problems, a computational saving of numerical simulation can be achieved by reducing degrees of freedom in the linear system. Appropriate reduction of degrees of freedom of the problem by controlling error of approximation is often handled via a \textit{posteriori} error analysis. Such a special assessment of errors provides basis for mesh refinement or unrefinement strategy. 
A \textit{posteriori} error estimator and \textit{Error Indicator} of elements is essential to determine the large error element which can be considered for refinement to achieve the accurate and efficient subsurface simulations. By knowing the assessment of error provided by estimators, one can control of discretization error and achieve the anticipated quality of the numerical solution.  


A posteriori error estimators for finite element method for elliptic boundary value problems started by the Babuska and Rheinbodt work in \cite{babuvska1978posteriori}. The main objective of a posteriori error estimation is to obtain the estimator that is close to the error in specific (e.g. energy) norm on each element \cite{ainsworth2011posteriori}. 
The literature on a \textit{posteriori} error estimates and adaptivity for mixed finite element approximation has highlighted several forms of estimates for mesh refinement strategy. To derive estimates for many finite element approximation, many researchers have proposed various methods of optimal a \textit{posteriori} error estimate in \cite{ainsworth2011posteriori, verfurth2013posteriori, estep2000estimating, larson2008posteriori}. In particularly, conforming mixed finite element method were explored in \cite{vohralik2010unified, wohlmuth1999residual, vohralik2007posteriori,zienkiewicz1987simple} as well as Discontinuous Galerkin (DG) method employing explicit error estimate in \cite{sun2007discontinuous} and implicit error estimate in \cite{riviere2003posteriori}, goal-oriented Discontinuous Petrov-Galerkin(DPG) \cite{keith2017goal} with applications \cite{fuentes2017using}, for nonconforming FEM in \cite{carstensen2007unifying, ern2015polynomial,ainsworth2005robust} such as Multiscale Mortar Mixed FEM  in \cite{pencheva2013robust, sun20052, wheeler2005posteriori, arbogast2014posteriori,arbogast2015posteriori}. However, to the best of our knowledge, a \textit{posteriori} error analysis has not been conducted for Enhanced Velocity Mixed FEM, which is practical method in adaptive setting for many subsurface applications \cite{amanbek2019adaptive,thomas2011enhanced, wheeler2002enhanced,singh2017adaptive,amanbek2018priori}.

For large domain with heterogeneous permeability $\mathbf{K}$, which varies over scales from mm to hundreds kms, it is computationally intensive task.  Direct discretization in fine scale of entire domain would involve high resolution permeability $\mathbf{K}$ of $\Omega$ that resulting a large and coupled system of equation. Such system solution would usually be CPU and Memory demanding.

Recent development for flow and transport in heterogeneous porous media, adaptive numerical homogenization using Enhanced Velocity Mixed Finite Element Method (EVMFEM) in \cite{amanbek2019adaptive,amanbek2018new,singh2017adaptive}, motivates us to study a \textit{posteriori} error estimates in the adaptive mesh refinement strategy. The use of adaptive numerical homogenization captures fine scale features in efficient way for heterogeneous porous media. This method can be successfully used for a number of subsurface engineering applications to achieve efficient and accurate numerical solution. 

%
%
%
%
%
%
%
%

The main purpose of this paper is to provide a \textit{posteriori} error estimates for single phase flow using EVMFEM. We show theoretical derivations of estimators and then numerical results for selected indicators to confirm theoretical upper bounds. Velocity error is taken into account to increase accuracy of velocity that are important in transport equations. We consider nested version of mesh discretization at the interface for numerical experiments. First, we derive the explicit residual-based a \textit{posteriori} error estimates for EVMFEM with saturation assumption. Second, we show the implicit error estimates without saturation assumption using a suitable post-processing of the finite element pressure approximation that is better indicator for mesh refinement. We show estimates in $L^2$ norm. For simplicity, the problem is considered with Dirichlet boundary conditions, but the result can be generalized.

The remainder of this paper is organized as follows. Section 2 outlines formulation of EVMFEM and preliminaries. Our a posteriori error estimates are presented in Section 3. In this section, we start by going through briefly preliminaries and useful inequalities, formulation and projections. Second, residual-based explicit error estimators were derived for the pressure and velocity errors. Then, residual-based estimators with smoothing were derived to improve the indicators of velocity error by using post-processed pressure such as the Arbogast and Chen postprocessing \cite{arbogast1995implementation}. Section 4 shows computational results. The conclusion is reported in Section 5.
\section{Model formulation} \label{modelformulation}
We start by giving the model formulation for the incompressible single-phase flow. For the convenience of reader we repeat the relevant material of domain decomposition method, discrete formulation with Enhanced Velocity from \cite{wheeler2002enhanced}. 
EVMFEM is a mass conservative and an efficient  domain decomposition method which deals with non-matching grids or multiblock grids \cite{wheeler2002enhanced}. This method is extensively used in many complex multicomponenet, multicomponent, multiphase flow and transport processes in porous media \cite{thomas2011enhanced}. EVMFEM approach is strongly mass conservative at the interfaces and impose strong continuity of fluxes between the subdomains. Earlier implementations \cite{wheeler2002enhanced,thomas2011enhanced} employed a solution approach where only the coarse and fine domain contributions to the stiffness-matrix (or Jacobian matrix) were taken, neglecting interface contributions. The load vector (or residuals); however, contains contributions from both the coarse and fine subdomains as well as the interface. This resulted in an increase in the number of non-linear iterations to achieve convergence, for a given tolerance, even for a linear flow and transport problem. In this work, we use a fully coupled variant of the original EVMFEM approach wherein the interface terms are properly accounted for in the stiffness-matrix construction resulting in reduced non-linear iterations (one for a linear system).

\subsection{Governing Equations of the Incompressible Flow }
For convenience of analysis, we consider the incompressible single phase flow model for pressure $p$ and the Darcy velocity $\textbf{u}$: 
\begin{align} 
\textbf{u} &=-\mathbf{K} \nabla p \qquad \text{in} \quad \Omega,  \label{eq:a} \\ 
\nabla \cdot \textbf{u} &=f \qquad \qquad \text{in} \quad \Omega, \label{eq:b}  \\ 
p &=g \qquad \qquad \text{on} \quad \partial \Omega   \label{eq:c} 
\end{align}
where $\Omega \in \mathbb{R}^d( d=2$ or $3$) is multiblock domain, $f \in L^2(\Omega)$  and  $\mathbf{K}$ is a symmetric, uniformly positive definite tensor representing the permeability divided by the viscosity with $L^{\infty}(\Omega)$ components, for some $0<k_{min}<k_{max} < \infty$
\begin{align}
k_{min} \xi ^T \xi \le \xi^T \mathbf{K}(x) \xi \le k_{max} \xi^T \xi \qquad \forall x \in \Omega \quad \forall \xi \in \mathbb{R}^d.
\end{align}

Let $\Omega$ be divided into a series of small subdomains. For simplicity, the Dirichlet boundary condition is considered as zero, i.e. $g=0$. To formulate in mixed variational form, Sobolev spaces are exploited and the following space is defined for flux in $\mathbb{R}^d$ as usual to be
$\textbf{V}=H({\rm div}; \Omega) =\{\mathbf{v}  \in \left(L^2(\Omega)\right)^d: \nabla \cdot \mathbf{v}  \in L^2(\Omega)\} $
and equipped with the norm $ \norm{\mathbf{v} }_V=\left(\norm{\mathbf{v} }^2+\norm{\nabla \cdot \mathbf{v} }^2\right)^{\frac{1}{2}} $
and for the pressure   the space is 
$W =L^2(\Omega)$ and the corresponding norm $\norm{w}_W=\norm{w}$.

We utilize standard notations. For subdomain $\zeta \subset \mathbb{R}^d$, the $L^2(\zeta)$ inner product (or duality pairing) and norm are denoted by $(\cdot, \cdot)_S$ and $\norm{\cdot}_{\zeta}$, respectively, for scalar and vector valued functions. Let $W^{m,p}$ be the standard Sobolev space of $m$-differentiable functions in $L^p(\zeta)$. Let $\norm{\cdot}_{m,\zeta}$ be norm of $H^{m}(\zeta) = W^{m,2}(\zeta)$ or $H^{m}(\zeta)$, where $\zeta$ and $m$ are omitted in case of $\zeta = \Omega$ and $m=0$ respectively, in other cases they are specified. We write $(\cdot, \cdot)$ for the $L^2(\zeta)$ or $\left(L^2(\zeta)\right)^d$ inner product, and $\langle \cdot, \cdot \rangle_{\partial \zeta}$ for duality pairing on boundaries and interfaces, where the pairing may be between two functions in $L^2$ or between elements of $H^{1/2}$ and $H^{-1/2}$, in either order.

Next, a weak variational form of the fluid flow problem $(\ref{eq:a})-(\ref{eq:c})$ is to find a pair $\textbf{u} \in \mathbf{V}$, $p \in W$
\begin{align} 
\left(\mathbf{K}^{-1}\textbf{u}, \mathbf{v} \right) -\left(p, \nabla \cdot \mathbf{v} \right) &= - \langle g, \mathbf{v} \cdot \nu \rangle_{\partial \Omega} \qquad & \forall \mathbf{v} \in \textbf{V} \label{eq:2_4} \\ 
\left(\nabla \cdot \textbf{u}, w \right) &=\left(f,w\right) \qquad \qquad &\forall w \in W \label{eq:2_5}  
\end{align}
where $\nu$ is the outward unit normal to $\partial \Omega$. 

\begin{figure}[H]
	\centering
	\begin{tikzpicture}[thick,scale=1.5] 
	
	\def\xa{-2}
	\def\xb{0}
	\def\xc{3}
	
	\def\ya{-2}
	\def\yb{-1}
	\def\yc{0}
	\def\yd{1}
	\def\ye{3}
	\def\yf{4}
	
	\coordinate (A1) at (\xa, \yd);
	\coordinate (A2) at (\xb, \yd);
	\coordinate (A3) at (\xb, \yf);
	\coordinate (A4) at (\xa, \yf);
	\coordinate (B1) at (\xa, \yb);
	\coordinate (B2) at (\xb, \yb);
	\coordinate (C1) at (\xb, \ya );
	\coordinate (C2) at (\xc, \ya);
	\coordinate (C3) at (\xc, \ye);
	\coordinate (C4) at (\xb, \ye);
	\coordinate (D1) at (\xb, \yc);
	\coordinate (D2) at (\xc, \yc);

	\draw[fill=blue!30,opacity=0.4] (A1) -- (A2) -- (A3) -- (A4)--cycle;
	\draw[fill={rgb:orange,1;yellow,2;pink,5},opacity=0.6] (C1) -- (C2) -- (D2) -- (D1)--cycle;
	\draw[fill={rgb:orange,1;yellow,2;pink,5},opacity=0.3] (D1) -- (D2) -- (C3) -- (C4)--cycle;
	\draw[fill={rgb:orange,1;yellow,1;red,2},opacity=0.4] (B1) -- (B2) -- (A2) -- (A1)--cycle;
	
	
	\draw[step=5mm,black,thick,dashed] (A1) grid (A3); 
	\draw[step=10mm,black, thick,dashed] (B1) grid (A2); 
	\draw[step=5mm,black,thick,dashed] (C1) grid (D2); 
	\draw[step=6mm,black, thick,dashed] (D1) grid (C3);

	\draw[black,line width=1.5mm,ultra thick] (B1) rectangle (A2);
	\draw[black,line width=1.5mm,ultra thick] (C1) rectangle (D2);
	\draw[black,line width=1.5mm,ultra thick] (D1) rectangle (C3);
	\draw[black,line width=1.5mm,ultra thick] (A1) rectangle (A3);
	
	\fill[red!50,ultra thick] (\xb-0.025,\yb) rectangle (\xb+0.025,\ye);
	\fill[red!50,ultra thick] (\xa,\yd-0.025) rectangle (\xb,\yd+0.025);
	\fill[red!50,ultra thick] (\xb,\yc-0.025) rectangle (\xc,\yc+0.025);

	\node [above right] at (B1) { $\Omega_1$};
	\node [above right] at (A1) {$\Omega_2$};
	\node [above right] at (D1) {$\Omega_3$};
	\node [above right] at (C1) {$\Omega_4$};
	\node [below left] at (C3) {$\mathcal{T}_h$};

	\draw[-latex,thick](\xc+1,\ye)node[above]{$\partial \Omega$} to[out=270,in=0] (\xc,\ye-1);
	\draw[-latex,thick](\xc+1,\yb)node[above]{$\Gamma$}  to[out=180,in=270] (\xc-1.4,\yc);	
	\end{tikzpicture}
	\caption{Illustration of a domain $\Omega$ with subdomains $\Omega_i$ and non-matching mesh discretization $\mathcal{T}_h$.}
\end{figure}
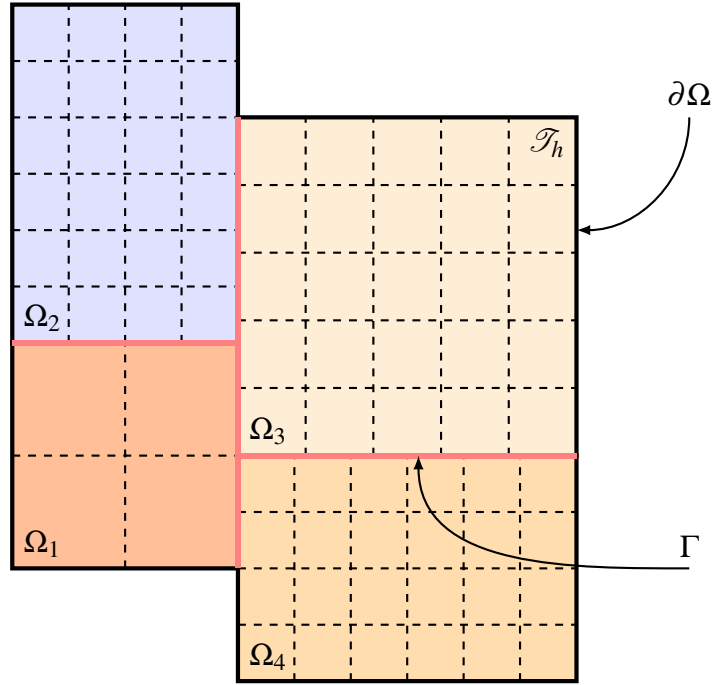

\subsubsection*{Discrete formulation}
We consider 
\begin{align*}
\Omega = \left( \bigcup_{i=1}^{N_b} \bar{\Omega}_i \right)^{o}, \qquad \Gamma_{i,j}=\partial \Omega_i \bigcap \partial \Omega_j, \quad \Gamma = \left( \bigcup^{N_b}_{i,j=1} \bar{\Gamma}_{i,j} \right)^{o}, \quad \Gamma_i = \Omega_i \bigcap \Gamma=\partial \Omega_j \setminus \partial \Omega.
\end{align*}
This implies that the domain is divided into $N_b$ subdomains, the interface between $i^{th}$ and $j^{th}$ subdomains($i \ne j$), the interior subdomain interface for $i^{th}$ subdomain and union of all such interfaces, respectively.

Let $\mathcal{T}_{h,i}$ be a conforming, quasi-uniform and rectangular partition of $\Omega_i$, $1 \le i \le N_b$, with maximal element diameter $h_i$. We then set $\mathcal{T}_{h}=\cup_{i=1}^{n} \mathcal{T}_{h,i} $ and denote $h$ the maximal element diameter in $\mathcal{T}_h$; note that $\mathcal{T}_h$ can be nonmatching as neighboring meshes $\mathcal{T}_{h,i}$ and $\mathcal{T}_{h,j}$ need not match on $\Gamma_{i,j}$. We assume that all mesh families are shape-regular.

We narrow our work to the regularly used Raviart-Thomas spaces of lowest order on rectangles for $d=2$ and bricks for $d=3$. The $RT_0$ spaces are defined for any element $T \in \mathcal{T}_{h}$ by the following spaces:
\begin{align*}
&\mathbf{V}_h(T)=\{\mathbf{v}=(v_1,v_2) \; \text{or} \; \mathbf{v}=(v_1,v_2,v_3): v_l=\alpha_l+\beta_l x_l: \alpha_l, \beta_l \in \mathbb{R};l=1,..d \} \quad \text{and} \\
&W_h(T) =\{w=\text{constant}\}
\end{align*}
In fact, a vector function in $\mathbf{V}_h$ can be determined uniquely by its normal components $\mathbf{v} \cdot \nu $ at midpoints of edges (in 2D) or face (in 3D) of $T$. The degrees of freedom of $\mathbf{v} \in \mathbf{V}_h(T)$ were created by these normal components. The degree of freedom for a pressure function $p \in W_h(T)$ is at center of $T$ and piecewise constant inside of $T$. 
The pressure finite element approximation space on $\Omega$ is taken to be as
\begin{equation*}
W_h(\Omega) =\{w\in L^2(\Omega): w \biggr\rvert_E \in W_h(T), \forall T \in \mathcal{T}_h \}
\end{equation*}

We first construct a velocity finite element approximation space on $\Omega$, which is different from the velocity space of the Multiscale Mortar Mixed FEM. Let us formulate $RT_0$ space on each subdomain $\Omega_i$ for partition $\mathcal{T}_h$
\begin{align*}
\mathbf{V}_{h,i}=\{\mathbf{v} \in H({\rm div};\Omega_i): \mathbf{v} \biggr\rvert_T \in \mathbf{V}_h(T), \forall T \in \mathcal{T}_{h,i} \}  \qquad i \in \{1,...n\}
\end{align*}
and then
\begin{align*}
\mathbf{V}_{h}= \bigoplus_{i=1}^n \mathbf{V}_{h,i}.
\end{align*}
Although the normal components of vectors in $\mathbf{V}_h$ are continuous between elements within each subdomains, the reader may see $\mathbf{V}_h$ is \textit{not} a subspace of $H({\rm div}; \Omega)$, because the normal components of the velocity vector may not match on subdomain interface $\Gamma$. To solve this issue, many researchers have proposed various methods such as Multiscale Mortar Mixed FEM \cite{arbogast2007multiscale}, Enhanced Velocity Mixed FEM  \cite{wheeler2002enhanced}, etc. In Mortar Multiscale Mixed FEM , the mortar finite element space on coarse grid was introduced to connect subdomains together using Lagrange multipliers to enforce weak continuity for flux across subdomains. On the other hand, the Enhanced Velocity Mixed FEM modifies the degree of freedom on $\Gamma$ to finer grids, which impose the strong flux continuity between subdomains. Let us define $\mathcal{T}_{h,i,j}$ as the intersection of the traces of $\mathcal{T}_{h,i}$ and $\mathcal{T}_{h,j}$, and let $\mathcal{T}^{\Gamma}_h=\bigcup_{1 \le i \le j \le N_b} \mathcal{T}_{h,i,j}$. We require that $\mathcal{T}_{h,i}$ and $\mathcal{T}_{h,j}$ need to align with the coordinate axes. Fluxes are constructed to match on each element $e \in \mathcal{T}^{\Gamma}_h$. We consider any element $T \in \mathcal{T}_{h,i}$ that shares at least one edge with the interface $\Gamma$, i.e., $T \cap \Gamma_{i,j} \ne \emptyset$, where $1 \le i,j \le N_b$ and $i \ne j$. Then newly defined interface grid introduces a partition of the edge of $T$. This partition may be extended into the element $T$ as shown in Figure \ref{fig:EVMFEM_Gamma}.

\begin{figure}[!tbp]
	\centering	
	\begin{tikzpicture}[thick,scale=1.5, dot/.style = {outer sep = +0pt, inner sep = +0pt, shape = circle, draw = black, label = {#1}},
	small dot/.style = {minimum size = 1pt, dot = {#1}},
	big dot/.style = {minimum size = 8pt, dot = {#1}},
	line join = round, line cap = round, >=triangle 45
	]
	
	\coordinate (A1) at (0, -1);
	\coordinate (A2) at (0, 1);
	\coordinate (A3) at (2, 1);
	\coordinate (A4) at (2, -1);
	
	\coordinate (B1) at (-2.6,-2);
	\coordinate (B2) at (-2.6, 1.5);
	\coordinate (B3) at (0,1.5);
	\coordinate (B4) at (0,-2);
	
	\coordinate (C1) at (-2.6,0);
	\coordinate (C2) at (0, 0);
	\coordinate (C3) at (2,0);
	
	\coordinate (D1) at (-2.6, -1);
	\coordinate (D2) at (-2.6,1);

	\fill[fill={rgb:orange,1;yellow,2;pink,5}, opacity=0.2] (A1) rectangle (A3); 
	
	\draw[very thick] (A1) -- (A2) -- (A3) -- (A4) -- cycle;
	
	\draw[fill={rgb:orange,1;yellow,2;blue,2},opacity=0.2, very thick] (B1) -- (B4)-- (C2)--(C1);
	\draw[fill={rgb:orange,2;yellow,1;green,1},opacity=0.2, very thick] (C1) -- (C2)-- (B3)--(B2);
	
	\draw[very thick] (B4) -- (B3);
	\draw[very thick] (B3) -- (B2);
	\draw[very thick] (C1) -- (C2);
	\draw[dotted, color =red] (C2) -- (C3);
	\draw[dotted, color =red] (A2) -- (D2);
	\draw[dotted, color =red] (A1) -- (D1);
	\draw[very thick] (B1) -- (B4);
	\draw[very thick] (B1) -- (B2);

	\node[font = \large, color =red] (e1)  at (0, 0.5) { \textbf{$\times$}};
	\node[font = \large, color =red] (e2)  at (0,-0.5) {$\times$};	
	\node[font = \large ] (e3)  at (1, -1) {$\times$};
	\node[font = \large] (e4)  at (2, 0) {$\times$};
	\node[font = \large] (e5)  at (1, 1) {$\times$};
	
	\node [above, font = \large] at (0,-2.7) { $\Gamma_{i,j}$};
	\node [above left, font = \large] at (0, 0.5) { $e_1$};
	\node [above left, font = \large] at (0,-0.5) { $e_2$};
	\node [right, font = \large] at (1, 0.5) { $T_1$};
	\node [right, font = \large] at (1,-0.5) { $T_2$};
	
	\end{tikzpicture}
	\caption{Degrees of freedom for the Enhanced Velocity space.}
	\label{fig:EVMFEM_Gamma}
\end{figure}
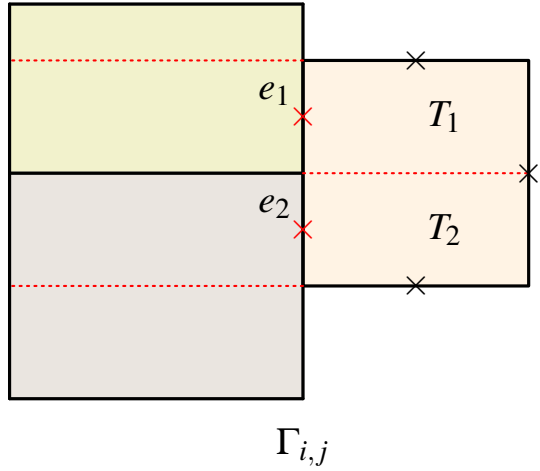

This new partitioning helps to construct fine-scale fluxes that is in $H(\textbf{ div}, \Omega)$. So we represent a basis function $\mathbf{v}_{T_k}$ in the $\mathbf{V}_h(T_k)$ space ($RT_0$) for given $T_k$ with the following way: 
\begin{align*}
\mathbf{v}_{T_k} \cdot \nu =
\begin{cases}
1, \qquad {\rm on} \; e_k \\
0, \qquad \rm other \; edges
\end{cases}
\end{align*}
i.e. a normal component $\mathbf{v}_{T_k} \cdot \nu $ equal to one on $e_k$ and zero on all other edges(faces) of $T_k$.
Let $\mathbf{V}^{\Gamma}_h$ be span of all such basis functions defined on all sub-elements induced the interface discretization $\mathcal{T}_{h, i,j}$. Thus, the enhanced velocity space $\mathbf{V}^*_h$ is taken to be as
\begin{align*}
\mathbf{V}^*_{h}= \bigoplus_{i=1}^n \mathbf{V}^0_{h,i} \bigoplus \mathbf{V}^{\Gamma}_{h} \cap H({\rm div}; \Omega).
\end{align*}
where $\mathbf{V}^0_{h,i} = \{ \mathbf{v} \in \mathbf{V}_{h,i} : \mathbf{v} \cdot \nu =0  \text{  on  } \Gamma_{i}\}$ is the subspace of $\mathbf{V}_{h,i}$. The finer grid flux allows to velocity approximation on the interface and then form the $H({\rm div}, \Omega)$ conforming velocity space. Some difficulties arise, however, in analysis of method and implementation of robust linear solver for such modification of $RT_0$ velocity space at all elements, which are adjacent to the interface $\Gamma$.
We now formulate the discrete variational form of equations $(\ref{eq:a})-(\ref{eq:c})$ as: Find  $\mathbf{u}_h \in \mathbf{V}^*_{h} $ and $p_h \in W_h$ such that
\begin{align} 
\left(K^{-1}\textbf{u}_h, \mathbf{v} \right) &=\left(p_h, \nabla \cdot \mathbf{v} \right)- \langle g, \mathbf{v} \cdot \nu \rangle_{\partial \Omega} \qquad & \forall \mathbf{v} \in \textbf{V}^*_h  \label{eq:EV_eq1} \\ 
\left(\nabla \cdot \textbf{u}_h, w \right) &=\left(f,w\right) \qquad \qquad &\forall w \in W_h\label{eq:EV_eq2}  
\end{align}

\section{Methodology of Error Estimate}\label{sec:method}

In this section, we derive error estimators for enhanced velocity mixed finite element discretization of elliptic problems. First, we briefly go through preliminaries and useful inequalities, formulation and projections. Second, residual based explicit error estimators were derived for pressure and velocity. Third, residual based estimators with smoothing were derived to improve the indicators of velocity error by using post-processed pressure such as the Arbogast and Chen postprocessing.

\subsection*{Representation of error}

We define 
\begin{align*}
e_{\mathbf{u}} = \mathbf{u}-\mathbf{u}_h \quad \text{ and } \quad  e_p =p - p_h .
\end{align*}
%
We introduce the bilinear form $\mathcal{A}(\mathbf{u}, p ; \mathbf{v}, w)$ defined as
\begin{align*} 
\mathcal{A}(\mathbf{u}, p ; \mathbf{v}, w) = \left(\mathbf{K}^{-1}\textbf{u}, \mathbf{v} \right) - \left(p, \nabla \cdot \mathbf{v} \right) + \mu \left(\nabla \cdot \textbf{u}, w \right)  
\end{align*}
where $\mu = 1$ or $\mu=-1$. We denote $\mathcal{A}(\cdot, \cdot)$ as $\mathcal{A}_s(\cdot, \cdot)$ when $\mu = -1$, which is a symmetric bilinear form and  $\mathcal{A}(\cdot, \cdot)$ as $\mathcal{A}_c(\cdot, \cdot)$ when $\mu = 1$, which is a nonsymmetric, but coercive, since $\mathcal{A}_c(\mathbf{v}, w; \mathbf{v}, w)=(\mathbf{K}^{-1} \mathbf{v}, \mathbf{v})$. Linear functional $L(\mathbf{v}, w)$ is defined as
\begin{align*} 
L(\mathbf{v}, w) = \mu (f, w) -\langle g, \mathbf{v} \cdot \nu \rangle 
\end{align*}
where $\mu = 1$ or $\mu=-1$.  Note that the solution does not depend on the choice of $\mu$.
Therefore, weak variational form of (\ref{eq:2_4})-(\ref{eq:2_5}) imply that $(\mathbf{u}, p) \in \mathbf{V} \times W$ satisfy 
\begin{align*} 
\mathcal{A}(\mathbf{u}, p ; \mathbf{v}, w) = L(\mathbf{v}, w)  \qquad \qquad (\mathbf{v}, w) \in \mathbf{V} \times W
\end{align*}
and the discrete variational formulation in (\ref{eq:EV_eq1})- (\ref{eq:EV_eq2}) implies that 
\begin{align} \label{eq:weakdiscreteform} 
\mathcal{A}(\mathbf{u}_h, p_h ; \mathbf{v}, w) = L(\mathbf{v}, w)  \qquad \qquad (\mathbf{v}, w) \in \mathbf{V}^*_h \times W_h
\end{align}
Then using above notations we obtain the residual equation
\begin{align}\label{eq:residualform} 
\mathcal{A}(e_{\mathbf{u}}, e_p ; \mathbf{v}, w) = L(\mathbf{v}, w) - \mathcal{A}(\mathbf{u}_h, p_h ; \mathbf{v}, w)   \qquad \qquad (\mathbf{v}, w) \in \mathbf{V} \times W
\end{align}
hold true  for each pair $(e_{\mathbf{u}}, e_p) \in \mathbf{V} \times W$. Thus, (\ref{eq:weakdiscreteform}) and (\ref{eq:residualform}) give the orthogonality condition
\begin{align}
\mathcal{A}(e_{\mathbf{u}}, e_p ; \mathbf{v}, w) = 0 \qquad \qquad (\mathbf{v}, w) \in \mathbf{V}^*_h \times W_h
\end{align}

\subsection*{Preliminaries and useful inequalities}

We assume the model problem is $H^2$-regular, in other words, there exists a positive C that depends on $\mathbf{K}$ and $\Omega$ such that
\begin{align} \label{eqn:ellipregularity}
\norm{p}_2 \le C \left(\norm{f}+\norm{g}_{\frac{3}{2}, \Gamma_D}\right)
\end{align}
We refer to the reader to \cite{gilbarg2015elliptic} for sufficient conditions for $H^2$-regularity.

We present below some of the approximation properties of the finite element spaces
\begin{align*}
(w-\hat{w}, w_h)=0  \quad \forall w_h \in W_h
\end{align*}

Furthermore, the following important approximation properties hold true. For all $T \in \mathcal{T}_h$, $e \in \mathcal{T}_{h,i}|_{\partial \Omega_i}$, and smooth enough $\mathbf{v}$ and $w$
\begin{align}
\norm{\mathbf{v} - \Pi \mathbf{v}}_T &\le C h_T \norm{\mathbf{v}}_{1, T} \\
\norm{w - \hat{w}} &\le C_P h_T \norm{w}_{1, T}
\end{align}
where $C_P=1/\pi$, if $T$ is convex. 
Above inequalities are standard $L^2$-projection approximation results and can be found in \cite{ciarlet2002finite}. 
In the analysis below we will use of trace inequalities
\begin{align}
\forall T \in \mathcal{T}_h, \; e \in \partial T, \; \norm{\phi}_{e} &\le C \left(h^{-1/2}_T\norm{\phi} + h^{1/2}_T\norm{\nabla \phi} \right) \; \phi \in H^1(T) \\
\forall T \in \mathcal{T}_h, \; e \in \partial T, \; \norm{\phi}_{1/2, e} &\le C \norm{\phi}_{1,T} \qquad \phi \in H^1(T) \\
\forall T \in \mathcal{T}_h, \; e \in \partial T, \; \norm{\mathbf{v}\cdot \nu}_{e} &\le C h^{-1/2}_T\norm{\mathbf{v}}_{T} \quad \mathbf{v} \in \mathbf{V}^*_h
\end{align}
and Young's inequality
\begin{align*}
2ab \le \varepsilon a^2 +\frac{1}{\varepsilon}b^2 
\end{align*}
We know that from original work \cite{wheeler2002enhanced} the projection operator $\Pi^*$ was introduced and was utilized for a \textit{priori} error analysis. For convenience of the reader, we repeat the relevant and brief definition. Thus, we denote by $\Pi^*$ the projection operator that maps $(H^1(\Omega))^d$ onto $\mathbf{V}^*_h$ that defined locally for any element $T \in \mathcal{T}_h$ and any $\mathbf{q} \in (H^1(T))^d$ such that for all $\mathbf{q} \in (H^1(T))^d$ 
\begin{align}
\langle \Pi^* \mathbf{q} \cdot \nu, 1 \rangle_e = \langle \mathbf{q} \cdot \nu, 1 \rangle_e \label{eq:def_proj_star}
\end{align}
where $e$ is either any edge in 2D (or face in 3D) of $T$ not lying on $\Gamma$  or an edge in 2D (or face in 3D) of a sub-element, $T_k$. Such projection is developed prior to conducting error analysis for a \textit{priori} and a \textit{posteriori} error estimates.  As can be seen in Figure \ref{fig:EVMFEM_Gamma}, $T_k$ has a common edge with the interface grid $\mathcal{T}^{\Gamma}$. According to divergence theorem, we have
\begin{align}
\left(\nabla \cdot (\Pi^*\mathbf{q} -\mathbf{q}), w \right) = 0 \qquad \qquad \forall w \in W_h
\end{align} 
We refer the reader to the original work \cite{wheeler2002enhanced} for more details. We want to scale it to local element. So by scaling argument we reach the following lemma
\begin{lem}
	Let $\mathbf{v} \in \left(H^1(T)\right)^d$ then $\exists$ C independent of $h$ such that
	\begin{align}
	\norm{\Pi^* \mathbf{v} -\mathbf{v}}_{T} \le C h_T \norm{\mathbf{v}}_{1,T} 
	\end{align}
\end{lem}

\begin{lem}
	Let $\mathbf{v} \in \left(H^1(T)\right)^d$ then $\exists$ C independent of $h$ such that
	\begin{align}
	\norm{(\Pi^* \mathbf{v} -\mathbf{v}) \cdot \nu}_{\partial T} \le C h^{1/2}_T \norm{\mathbf{v}}_{1,T}
	\end{align}
\end{lem}
\begin{proof}
	\begin{align*}
	\norm{(\Pi^* \mathbf{v} -\mathbf{v}) \cdot \nu}_{\partial T} \le C h^{-1/2}_T \norm{\Pi^* \mathbf{v} -\mathbf{v}}_T \le C h^{-1/2}_T h_T \norm{\mathbf{v}}_{1,T} =C h^{1/2}_T \norm{\mathbf{v}}_{1,T} 
	\end{align*}
\end{proof}

\newpage
\subsection{Explicit Residual-based Error Estimators}

In this section, we derive upper bounds on the local error. It is also called explicit estimators as they involve the input data and computed numerical solution without solving extra sub-problems. We are not interested
in computing the constants in the error estimates and take the boundary condition as $g=0$.

\subsubsection{Estimates for Pressure}
\begin{thm}
	There exists a constant $C$ independent of $h$ such that
	\begin{equation}
	\norm{e_p}^2 \le C \{\sum_{T\in \mathcal{T}_h} (\tilde{\zeta}_P +\tilde{\zeta}_{R}) +\tilde{\zeta}_{EV}\}
	\end{equation}
	where, for all $T \in \mathcal{T}_h $
	\begin{align*}
	\tilde{\zeta}_P &= \norm{\mathbf{K}^{-1}\mathbf{u}_h + \nabla p_h}^2_T  h^2_T \\
	\tilde{\zeta}_{R, h} &= \norm{f - \nabla \cdot \mathbf{u}_h}^2_T h^2_T \\
	\tilde{\zeta}_{EV} &= \sum_{e \in \mathcal{T}^{\Gamma}_{h}} \norm{ \llbracket p_h \rrbracket}^2_{e} h_{T}
	\end{align*}
\end{thm}

\begin{proof}
	We consider a duality argument to derive bounds. Let $\varphi$ be the solution of the auxiliary problem
	\begin{align} 
	-\nabla \cdot \mathbf{K} \nabla \varphi &= e_p \qquad \text{in} \quad \Omega,   \\ 
	\varphi &=0 \qquad \qquad \text{on} \quad \partial \Omega.  
	\end{align}
	By the elliptic regularity assumption \ref{eqn:ellipregularity} implies that
	\begin{align*}
	\norm{\varphi}_2 \le C\norm{e_p}
	\end{align*}
	
	Let $\mathbf{v} = -\mathbf{K} \nabla \varphi$ then
	\begin{align*} 
	\mathcal{A}_s(\mathbf{v}, \varphi ; \mathbf{\tilde{v}}, \tilde{w}) = \left(\mathbf{K}^{-1}\textbf{v}, \mathbf{\tilde{v}} \right) - \left(\varphi, \nabla \cdot \mathbf{\tilde{v}} \right)- \left(\nabla \cdot \textbf{u}, \tilde{w}\right) =-(e_p, \tilde{w})
	\end{align*}
	
	Then
	\begin{align*} 
	\norm{e_p}^2 &= - \mathcal{A}_s(\mathbf{v}, \varphi; e_{\mathbf{u}}, e_p) = - \mathcal{A}_s(e_{\mathbf{u}}, e_p; \mathbf{v}, \varphi)= - \mathcal{A}_s(e_{\mathbf{u}}, e_p; \mathbf{v} - \Pi^*\mathbf{v} , \varphi -\hat{\varphi}) =  \\
	&=-\sum_{T\in \mathcal{T}_h} \{ \left(\mathbf{K}^{-1}e_{\mathbf{u}}, \mathbf{v} - \Pi^*\mathbf{v} \right)_{T} - \left(e_p, \nabla \cdot (\mathbf{v} - \Pi^*\mathbf{v} )\right)_{T} - \left(\nabla \cdot e_{\mathbf{u}}, \varphi -\hat{\varphi} \right)_{T} \}=\\
	&=-\sum_{T\in \mathcal{T}_h} \{ \left(\mathbf{K}^{-1}\mathbf{u}, \mathbf{v} - \Pi^*\mathbf{v} \right)_{T} - \left(\mathbf{K}^{-1}\mathbf{u}_h, \mathbf{v} - \Pi^*\mathbf{v} \right)_{T}  - \left(p, \nabla \cdot (\mathbf{v} - \Pi^*\mathbf{v} ) \right)_{T} \\
	&+ \left(p_h, \nabla \cdot \mathbf{v} - \Pi^*\mathbf{v} \right)_{T} -\left(\nabla \cdot \mathbf{u}, \varphi -\hat{\varphi} \right)_{T} + \left(\nabla \cdot \mathbf{u}_h, \varphi -\hat{\varphi} \right)_{T} \}= \\
	&= \sum_{T\in \mathcal{T}_h} \{  \left(\mathbf{K}^{-1}\mathbf{u}_h, \mathbf{v} - \Pi^*\mathbf{v} \right)_{T}  - \left(p_h, \nabla \cdot (\mathbf{v} - \Pi^*\mathbf{v} ) \right)_{T} + \left(f, \varphi -\hat{\varphi} \right)_{T} \\
	&- \left(\nabla \cdot \mathbf{u}_h, \varphi -\hat{\varphi} \right)_{T}   \}
	\end{align*}
	
	By Green's formula, $$\left(p_h, \nabla \cdot (\mathbf{v} - \Pi^*\mathbf{v} ) \right)_{\Omega_i}=-\left(\nabla p_h, \mathbf{v} - \Pi^*\mathbf{v}  \right)_{\Omega_i} + \langle p_h, \left( \mathbf{v} - \Pi^*\mathbf{v} \right) \cdot  \nu \rangle_{\partial \Omega_i}$$
	we obtain
	\begin{align*} 
	\norm{e_p}^2 &= \sum_{T\in \mathcal{T}_h} \left\{  \left(\mathbf{K}^{-1}\mathbf{u}_h + \nabla p_h, \mathbf{v} - \Pi^*\mathbf{v} \right)_{T} + \left(f - \nabla \cdot \mathbf{u}_h, \varphi -\hat{\varphi} \right)_{T} \right\}-  \sum_{i=1}^{n}\langle p_h, \left( \mathbf{v} - \Pi^*\mathbf{v} \right) \cdot  \nu_i \rangle_{\partial \Omega_i}
	\end{align*}
	
	We use the Cauchy-Schwarz inequality and approximation properties
	\begin{align*} 
	\norm{e_p}^2 &\le \sum_{T \in \mathcal{T}_h} \left(  \norm{\mathbf{K}^{-1}\mathbf{u}_h + \nabla p_h}_T \norm{\mathbf{v} - \Pi^*\mathbf{v}}_{T} + \norm{f - \nabla \cdot \mathbf{u}_h}_T \norm{\varphi -\hat{\varphi}}_{T} \right)\\
	& - \sum_{T \in \mathcal{T}_h} \langle p_h, \left( \mathbf{v} - \Pi^*\mathbf{v} \right) \cdot  \nu \rangle_{\partial T} \\ 
	&\le C \sum_{T \in \mathcal{T}_h} \left(  \norm{\mathbf{K}^{-1}\mathbf{u}_h + \nabla p_h}_T  h_T \norm{\mathbf{v}}_{1,T} + \norm{f - \nabla \cdot \mathbf{u}_h}_T h_T \norm{\varphi}_{1,T}\right) \\
	& + \frac{1}{2}\sum_{e \in \mathcal{T}^{\Gamma}_{h}, e \in \partial T_i \cup \partial T_j} \norm{\llbracket p_h \rrbracket}_{e} \norm{\left( \mathbf{v} - \Pi^*\mathbf{v} \right) \cdot  \nu }_{e} \\
	&\le C \sum_{T \in \mathcal{T}_h} \left(  \norm{\mathbf{K}^{-1}\mathbf{u}_h + \nabla p_h}_T  h_T \norm{\varphi}_{2,T} + \norm{f - \nabla \cdot \mathbf{u}_h}_T h_T \norm{\varphi}_{1,T} \right) \\
	&+ C\sum_{e \in \mathcal{T}^{\Gamma}_{h}, e \in \partial T_i \cup \partial T_j} \norm{\llbracket p_h \rrbracket}_{e}h^{\frac{1}{2}}_{T}\norm{\mathbf{v} }_{1, T} \\
	&\le C \sum_{T \in \mathcal{T}_h} \left(  \norm{\mathbf{K}^{-1}\mathbf{u}_h + \nabla p_h}_T  h_T \norm{e_p}_{T} + \norm{f - \nabla \cdot \mathbf{u}_h}_T h_T \norm{e_p}_{T} \right) \\
	&+ C \sum_{e \in \mathcal{T}^{\Gamma}_{h}, e \in \partial T_i \cup \partial T_j} \norm{\llbracket p_h \rrbracket}_{e} h^{\frac{1}{2}}_{T}\norm{\mathbf{v} }_{1,T} \\
	&\le C \sum_{T \in \mathcal{T}_h} \left(  \norm{\mathbf{K}^{-1}\mathbf{u}_h + \nabla p_h}_T  h_T \norm{e_p}_{T} + \norm{f - \nabla \cdot \mathbf{u}_h}_T h_T \norm{e_p}_{T} \right) \\ 
	&+ C \sum_{e \in \mathcal{T}^{\Gamma}_{h}, e \in \partial T_i \cup \partial T_j} \norm{\llbracket p_h \rrbracket}_{e} h^{\frac{1}{2}}_{T}\norm{e_p}_{T}
	\end{align*}
\end{proof}

\subsubsection{Estimates for Velocity with Saturation Assumption}
In order to get bounds on velocity error ($e_{\mathbf{u}}$), we employ a saturation assumption. Let be $\mathbf{V}^*_{h_f}, W_{h_f}$ be the finite element approximation spaces which corresponds to refinement of $\mathcal{T}_h$, where $h_f = h/m $ for $m \ge 2$. This implies that $\mathbf{V}^*_{h} \subset \mathbf{V}^*_{h_f} $ and $W_{h} \subset W_{h_f}$. A \textit{priori} error estimates from \cite{wheeler2002enhanced} allows us to employ the following saturation assumption.
\subsection*{Saturation assumption}
There exist constant $\beta < 1$,$\alpha < 1$ and $h_f = h/m $ for $m \ge 2$, $m \in \mathbb{N}$   such that

\begin{align} 
\norm{\textbf{u}-\mathbf{u}_{h_f}} &\le \beta \norm{\textbf{u}-\mathbf{u}_{h}} \label{eqn:saturationassumptionu} \\
\norm{p-p_{h_f}} &\le \alpha \norm{p-p_h} \label{eqn:saturationassumptionp}
\end{align}

\begin{lem}[]\label{lem:saturation}
	Let $\mathbf{u} \in \mathbf{H}\left({\rm div}, \Omega \right)$ be the exact flux defined by Eqns $1-3$ with $g=0$. Let $\mathbf{u}_h , \mathbf{u}_{h_f} \in \mathbf{L}^2\left(\Omega\right)$ be arbitrary and $h_f = h/m $ for $m \ge 2$, $m \in \mathbb{N}$. If there exist $\beta \in (0, 1)$ such that
	\begin{equation}
	\norm{\mathbf{u}-\mathbf{u}_{h_f}} \le \beta \norm{\mathbf{u}-\mathbf{u}_h},
	\end{equation}
	then, \vspace{-1em}
	\begin{equation}
	\frac{1}{1+\beta}\norm{\mathbf{u}_h-\mathbf{u}_{h_f}} \le \norm{\mathbf{u}-\mathbf{u}_h} \le \frac{1}{1-\beta}\norm{\mathbf{u}_h-\mathbf{u}_{h_f}}.
	\end{equation}	
\end{lem}
The proof of lemma \ref{lem:saturation}  is straightforward by using the triangle inequality.


We write $\mathbf{u}_{h_f} \in \mathbf{V}^*_{h_f}$ and $p_{h_f} \in W_{h_f}$ that are the enhanced velocity mixed finite element solution of equations (\ref{eq:EV_eq1}) - (\ref{eq:EV_eq2}) and let
\begin{align}
\tilde{e}_{\mathbf{u}} = \mathbf{u}_{h_f}-\mathbf{u}_h , \qquad \tilde{e}_{p} = p_{h_f}-p_h  
\end{align}
We have that $(\tilde{e}_{\mathbf{u}}, \tilde{e}_p) \in \mathbf{V}^*_{h, f} \times W_{h, f}$ which satisfy the residual equation
\begin{align}
\mathcal{A}(\tilde{e}_{\mathbf{u}}, \tilde{e}_p ; \mathbf{\tilde{v}}_h, \tilde{w}_h)=L(\mathbf{\tilde{v}}_h, \tilde{w}_h)-	\mathcal{A}(\mathbf{u}_h, p_h ; \mathbf{\tilde{v}}_h, \tilde{w}_h) \qquad \forall (\mathbf{\tilde{v}}_h, \tilde{w}_h) \in \mathbf{V}^*_{h_f} \times W_{h_f} 
\end{align}
and the orthogonality condition
\begin{align}
\mathcal{A}(\tilde{e}_{\mathbf{u}}, \tilde{e}_p ; \mathbf{v}_h, w_h)=0 \qquad \forall (\mathbf{v}_h, w_h) \in \mathbf{V}^*_h \times W_h 
\end{align}

\begin{thm}
	Assume that the saturations assumptions (\ref{eqn:saturationassumptionu}) and (\ref{eqn:saturationassumptionp}) hold. Then there exists a constant $C$ independent of $h$ such that
	\begin{align*}
	\norm{e_{\mathbf{u}}}^2 \le C \left(\sum_{T\in \mathcal{T}_h} \{\zeta_P + \zeta_{R} \} +\zeta_{EV} \right)
	\end{align*}
	where, for all $T \in \mathcal{T}_h $
	\begin{align*}
	\zeta_P &= \norm{\mathbf{K}^{-1}\mathbf{u}_h + \nabla p_h}^2_T  \\
	\zeta_{R, h} &= \norm{f - \nabla \cdot \mathbf{u}_h}^2_T h^2_T \\
	\zeta_{EV}&=\sum_{e \in \mathcal{T}^{\Gamma}_{h}} \norm{\llbracket p_h \rrbracket}^2_{e} h^{-1}_e
	\end{align*}
\end{thm}
\begin{proof}
	It is enough to bound $\tilde{e}_{\mathbf{u_h}}$, since the saturation assumption gives the following bound
	\begin{align}
	\norm{e_{\mathbf{u}}} \le \frac{1}{1-\beta} \norm{\tilde{e}_{\mathbf{u_h}}}
	\end{align}
	
	\begin{align*}
	\norm{\mathbf{K}^{-\frac{1}{2}}\tilde{e}_{\mathbf{u_h}}}^2 &= \mathcal{A}(\tilde{e}_{\mathbf{u}}, \tilde{e}_p ; \tilde{e}_{\mathbf{u}}, \tilde{e}_p) = \mathcal{A}(\tilde{e}_{\mathbf{u}}, \tilde{e}_p ; \tilde{e}_{\mathbf{u}}-\Pi^*\tilde{e}_{\mathbf{u}}, \tilde{e}_p) \\
	&=L(\tilde{e}_{\mathbf{u}}-\Pi^*\tilde{e}_{\mathbf{u}}, \tilde{e}_p)-\mathcal{A}(\mathbf{u}_h, p_h ; \tilde{e}_{\mathbf{u}}-\Pi^*\tilde{e}_{\mathbf{u}}, \tilde{e}_p) \\
	&= -\sum_{T\in \mathcal{T}_h} \left \{ \left(\mathbf{K}^{-1}\mathbf{u}_h, \tilde{e}_{\mathbf{u}}-\Pi^*\tilde{e}_{\mathbf{u}} \right)_{T} - \left(p_h, \nabla \cdot (\tilde{e}_{\mathbf{u}}-\Pi^*\tilde{e}_{\mathbf{u}} )\right)_{T} + \left(\nabla \cdot \mathbf{u}_h, \tilde{e}_p \right)_{T} \right \} \\ &- (f, \tilde{e}_p )  
	\end{align*}
	Using Green's formula 
	\begin{align*}
	\norm{\mathbf{K}^{-\frac{1}{2}}\tilde{e}_{\mathbf{u_h}}}^2 =&-\sum_{T\in \mathcal{T}_h} \left \{ \underbrace{ \left(\mathbf{K}^{-1}\mathbf{u}_h+\nabla p_h, \tilde{e}_{\mathbf{u}}-\Pi^*\tilde{e}_{\mathbf{u}} \right)_{T}}_{\mathbb{T}_1} + \underbrace{ \left(\nabla \cdot \mathbf{u}_h-f, \tilde{e}_p \right)_{T}}_{\mathbb{T}_2} \right \}  \\
	&-\underbrace{ \sum_{i=1}^n  \langle p_h, (\tilde{e}_{\mathbf{u}}-\Pi^*\tilde{e}_{\mathbf{u}} ) \cdot \nu_i\rangle_{\Gamma_i}}_{\mathbb{T}_3} 
	\end{align*}
	We treat three terms in the equation separately.
	\begin{align}
	\mathbb{T}_1 \le | \left(\mathbf{K}^{-1}\mathbf{u}_h+\nabla p_h, \tilde{e}_{\mathbf{u}}-\Pi^*\tilde{e}_{\mathbf{u}} \right)_{T}| \le C\left( \frac{1}{4\varepsilon_1}\norm{\mathbf{K}^{-1}\mathbf{u}_h+\nabla p_h}^2_T + \varepsilon_1 \norm{\tilde{e}_{\mathbf{u}}}^2_T \right)
	\end{align}
	Similarly for second term, we obtain
	\begin{align*}
	\mathbb{T}_2 &\le | \left(\nabla \cdot \mathbf{u}_h-f, \tilde{e}_p-P_h\tilde{e}_p \right)_{T}| \le \norm{\nabla \cdot \mathbf{u}_h-f}_T C h_T \norm{\nabla \tilde{e}_p}_T \\
	& \le C \norm{\nabla \cdot \mathbf{u}_h-f}_T h_T (1+\alpha)\norm{\nabla e_p}_T \\
	& \le C \norm{\nabla \cdot \mathbf{u}_h-f}_T h_T \norm{-\mathbf{K}^{-1}\mathbf{u}+\mathbf{K}^{-1}\mathbf{u}_h-\mathbf{K}^{-1}\mathbf{u}_h-\nabla p_h}_T \\
	& \le C^2 \norm{\nabla \cdot \mathbf{u}_h-f}^2_T h^2_T+\frac{1}{4}\norm{\mathbf{K}^{-1}\mathbf{u}-\mathbf{K}^{-1}\mathbf{u}_h}^2+\frac{1}{4}\norm{\mathbf{K}^{-1}\mathbf{u}_h+\nabla p_h}^2_T 
	\end{align*}
	
	\begin{align*}
	\mathbb{T}_3 &\le \sum_{e \in \mathcal{T}^{\Gamma}_{h}} |\langle \llbracket p_h \rrbracket, (\tilde{e}_{\mathbf{u}}-\Pi^*\tilde{e}_{\mathbf{u}} ) \cdot \nu_i\rangle_{e}| \le C \sum_{e \in \mathcal{T}^{\Gamma}_{h}} \norm{[p_h]}_e \norm{\tilde{e}_{\mathbf{u}} \cdot \nu_i}_e \\
	&\le  C \sum_{e \in \mathcal{T}^{\Gamma}_{h}, T \cap e \ne \emptyset} \norm{\llbracket p_h \rrbracket}_e h^{-1/2} \norm{\tilde{e}_{\mathbf{u}}}_T  \\
	&\le \frac{C^2}{4\varepsilon_3} \sum_{e \in \mathcal{T}^{\Gamma}_{h}}  \norm{\llbracket p_h \rrbracket}^2_e h^{-1}+ \varepsilon_3 \sum_{T \in \mathcal{T}^{\Gamma}_h(\Omega^*)} \norm{\tilde{e}_{\mathbf{u}}}^2_T
	\end{align*}
	Combining all three terms for small enough $\varepsilon_1$ and $\varepsilon_3$ yields
	\begin{align*}
	\norm{\tilde{e}_{\mathbf{u_h}}}^2 \le C \left \{   \sum_{T\in \mathcal{T}_h} \norm{\mathbf{K}^{-1}\mathbf{u}_h+\nabla p_h}^2_T + \norm{\nabla \cdot \mathbf{u}_h-f}^2_T h^2_T\right\} +C \sum_{e \in \mathcal{T}^{\Gamma}_{h}} \norm{\llbracket p_h \rrbracket}^2_e h^{-1}_e 
	\end{align*} 
\end{proof}
\subsection{Lower Bound}

\begin{thm}
	Assume $f$ is polynomial with degree $m$. There exists a constant $C$ independent of $h$ such that
	\begin{align}
	\tilde{\zeta}_{P} + \zeta_R \le C\left( \norm{p-p_h} + \norm{\mathbf{u}-\mathbf{u}_h} h_T \right) \label{ineq:lowerbound}\\
	\zeta_{EV} \le C\left( \norm{p-p_h} + \norm{\mathbf{u}-\mathbf{u}_h} h_T \right)
	\end{align}
\end{thm}
\begin{proof}
	We have made use of a bubble function argument as it has been shown in \cite{wheeler2005posteriori, carstensen1997posteriori}
	\begin{align*}
	\tilde{\zeta}_P = \norm{\mathbf{K}^{-1}\mathbf{u}_h + \nabla p_h}_T h_T \le C\Big\{ \norm{e_p}_T+ \norm{e_{\mathbf{u}}}_Th_T\Big\}
	\end{align*}
	By applying the element bubble function technique we obtain 
	\begin{align*}
	\norm{f - \nabla \cdot \mathbf{u}_h}_T h_T \le C h_T\norm{\nabla \cdot \left( \mathbf{u}- \mathbf{u}_h \right)}_T \le C \norm{\mathbf{u} - \mathbf{u}_h}_T =\norm{e_{\mathbf{u}}}_T
	\end{align*}
	From this we conclude Inequality (\ref{ineq:lowerbound}).
	Note that $\norm{\llbracket p_h \rrbracket}_{e} \le \norm{p-p^{+}_h }_{e}+\norm{ p-p^{-}_h}_{e}$. According to the trace inequality we get 
	\begin{align*}
	\norm{ p - p_h }_{e}  &\le C \left(h_T^{-1/2}\norm{p-p_h}+h_T^{1/2}\norm{\nabla (p-p_h)}\right) \\
	&\le C \left(h_T^{-1/2}\norm{p-p_h}+h_T^{1/2}\norm{\mathbf{K}^{-1}\mathbf{u}_h + \nabla p_h}+h_T^{1/2}\norm{\mathbf{K}^{-1}(\mathbf{u}-\mathbf{u}_h)}\right) \\
	&\le C \left(h_T^{-1/2}\norm{p-p_h}+h_T^{1/2}\norm{\mathbf{K}^{-1}\mathbf{u}_h + \nabla p_h}+h_T^{1/2}\norm{\mathbf{K}^{-1}(\mathbf{u}-\mathbf{u}_h)}\right) \\
	& \le C \left(h_T^{-1/2}\norm{e_p}+h_T^{1/2}\norm{e_{\mathbf{u}}}\right)
	\end{align*}
	using above inequality for $\tilde{\zeta}_P$.
\end{proof}

\newpage
\subsection{Residual-based and Smoothing Estimators with Postprocessing}
In this section, we show a general a \textit{posteriori} error estimate in $L^2$ norm using suitable a polynomial functions $s$, which is obtained by the Arbogast and Chen postprocessing of $p_h$ \cite{arbogast1995implementation}. This is discussed in Section \ref{section:postprocessing}. The advantage of using the postprocessed values ($s$) lies in the fact that it leads to better indicator for pressure and flux error in the element.  

\subsubsection{Estimates for Pressure}
\begin{thm}
	Let $\tilde{p}_h \in H^1(\mathcal{T}_{h, i})$ and $s \in H^1(\Omega_i)$. There exists a constant $C$ independent of $h$ such that
	\begin{equation}
	\norm{e_p}^2 \le C \sum_{T\in \mathcal{T}_h} \Big\{\tilde{\eta}_P +\tilde{\eta}_{R} +\tilde{\eta}_{NC} \Big \} 
	\end{equation}
	where, for all $T \in \mathcal{T}_h $
	\begin{align*}
	\tilde{\eta}_P &= \norm{\mathbf{K}^{-1}\mathbf{u}_h + \nabla s}^2_T  h^2_T \\
	\tilde{\eta}_{R, h} &= \norm{f - \nabla \cdot \mathbf{u}_h}^2_T h^2_T \\
	\tilde{\eta}_{NC} &= \norm{\nabla(s-\tilde{p}_h)}^2_{T} h^2_T  \\
	\end{align*}
\end{thm}

\begin{proof}
	We consider a duality argument to derive bounds. Let $\varphi$ be the solution of the auxiliary problem
	\begin{align} 
	-\nabla \cdot K \nabla \varphi &= e_p \qquad \text{in} \quad \Omega,   \\ 
	\varphi &=0 \qquad \qquad \text{on} \quad \partial \Omega.  
	\end{align}
	By the elliptic regularity assumption,
	\begin{align}
	\norm{\varphi}_2 \le C\norm{e_p} \label{eqn:ellipreg}
	\end{align}
	
	Let $\mathbf{v} = -\mathbf{K} \nabla \varphi$ then
	\begin{align*} 
	\mathcal{A}(\mathbf{v}, \varphi ; \mathbf{\tilde{v}}, \tilde{w}) = -\sum^2_{i=1} \{ \left(\mathbf{K}^{-1}\textbf{v}, \mathbf{\tilde{v}} \right)_{\Omega_i} - \left(\varphi, \nabla \cdot \mathbf{\tilde{v}} \right)_{\Omega_i} + \left(\nabla \cdot \textbf{u}, \tilde{w}\right)_{\Omega_i}   \}=-(e_p, \tilde{w})
	\end{align*}
	Then
	\begin{align*} 
	\norm{e_p}^2 &= - \mathcal{A}(\mathbf{v}, \varphi; e_{\mathbf{u}}, e_p) = - \mathcal{A}(e_{\mathbf{u}}, e_p; \mathbf{v}, \varphi)= - \mathcal{A}(e_{\mathbf{u}}, e_p; \mathbf{v} - \Pi^*\mathbf{v} , \varphi -\hat{\varphi}) =  \\
	&=-\sum_{T\in \mathcal{T}_h} \{ \left(\mathbf{K}^{-1}e_{\mathbf{u}}, \mathbf{v} - \Pi^*\mathbf{v} \right)_{T} - \left(e_p, \nabla \cdot (\mathbf{v} - \Pi^*\mathbf{v} )\right)_{T} + \left(\nabla \cdot e_{\mathbf{u}}, \varphi -\hat{\varphi} \right)_{T} \}=\\
	&=-\sum_{T\in \mathcal{T}_h} \{ \left(\mathbf{K}^{-1}\mathbf{u}, \mathbf{v} - \Pi^*\mathbf{v} \right)_{T} - \left(\mathbf{K}^{-1}\mathbf{u}_h, \mathbf{v} - \Pi^*\mathbf{v} \right)_{T}  - \left(p, \nabla \cdot (\mathbf{v} - \Pi^*\mathbf{v} ) \right)_{T} \\
	&+ \left(p_h, \nabla \cdot \mathbf{v} - \Pi^*\mathbf{v} \right)_{T} + \left(\nabla \cdot \mathbf{u}, \varphi -\hat{\varphi} \right)_{T} - \left(\nabla \cdot \mathbf{u}_h, \varphi -\hat{\varphi} \right)_{T} \}= \\
	&= \sum_{T\in \mathcal{T}_h} \{ - \left(\mathbf{K}^{-1}\mathbf{u}_h, \mathbf{v} - \Pi^*\mathbf{v} \right)_{T}  + \left(p_h, \nabla \cdot (\mathbf{v} - \Pi^*\mathbf{v} ) \right)_{T} + \left(f-\nabla \cdot \mathbf{u}_h, \varphi -\hat{\varphi} \right)_{T} \} \\
	&= \sum_{T\in \mathcal{T}_h} \{  -\left(\mathbf{K}^{-1}\mathbf{u}_h, \mathbf{v} - \Pi^*\mathbf{v} \right)_{T}  + \left(\tilde{p}_h, \nabla \cdot (\mathbf{v} - \Pi^*\mathbf{v} ) \right)_{T} + + \left(f-\nabla \cdot \mathbf{u}_h, \varphi -\hat{\varphi} \right)_{T} \} 
	\end{align*}
	
	By Green's formula, $$\left(\tilde{p}_h, \nabla \cdot (\mathbf{v} - \Pi^*\mathbf{v} ) \right)_{\Omega_i}=-\left(\nabla \tilde{p}_h, \mathbf{v} - \Pi^*\mathbf{v}  \right)_{\Omega_i} + \langle \tilde{p}_h, \left( \mathbf{v} - \Pi^*\mathbf{v} \right) \cdot  \nu \rangle_{\partial \Omega_i}$$
	we obtain after using $s \in H^1(\Omega_i)$
	\begin{align*} 
	\norm{e_p}^2 &= \sum_{T\in \mathcal{T}_h} \Big\{  \left(\mathbf{K}^{-1}\mathbf{u}_h + \nabla s, \mathbf{v} - \Pi^*\mathbf{v} \right)_{T} + \left(f - \nabla \cdot \mathbf{u}_h, \varphi -\hat{\varphi} \right)_{T} \\
	&+\left(\nabla(s- \tilde{p}_h), \mathbf{v} - \Pi^*\mathbf{v}  \right)_{T} 
	- \langle \tilde{p}_h, \left( \mathbf{v} - \Pi^*\mathbf{v} \right) \cdot  \nu \rangle_{\partial T}\Big\}
	\end{align*}
	
	We use the Cauchy-Schwarz inequality and approximation properties
	\begin{align*} 
	\norm{e_p}^2 &\le \sum_{T \in \mathcal{T}_h}  \{ \norm{\mathbf{K}^{-1}\mathbf{u}_h + \nabla s}_T \norm{\mathbf{v} - \Pi^*\mathbf{v}}_{T} + \norm{f - \nabla \cdot \mathbf{u}_h}_T \norm{\varphi -\hat{\varphi}}_{T} \\ 
	&+\norm{\nabla (s-\tilde{p}_h)}_T \norm{\mathbf{v} - \Pi^*\mathbf{v}}_{T}  -\sum_{i=1}^{n} \langle \tilde{p}_h, \left( \mathbf{v} - \Pi^*\mathbf{v} \right) \cdot  \nu_i \rangle_{\Gamma_i}  \le \\
	&\le C \sum_{T \in \mathcal{T}_h} \left(  \norm{\mathbf{K}^{-1}\mathbf{u}_h + \nabla s}_T  h_T \norm{\mathbf{v}}_{1,T} + \norm{f - \nabla \cdot \mathbf{u}_h}_T h_T \norm{\varphi}_{1,T}\right)  \\
	&+\norm{\nabla (s-\tilde{p}_h)}_T\norm{\mathbf{v} }_{1,T}  + \sum_{e \in \mathcal{T}_{h}^{\Gamma}} \langle \llbracket \tilde{p}_h \rrbracket, \left( \mathbf{v} - \Pi^*\mathbf{v} \right) \cdot  \nu_i \rangle_{\Gamma_i} \\
	&\le C \sum_{T \in \mathcal{T}_h}  \norm{\mathbf{K}^{-1}\mathbf{u}_h + \nabla s}_T  h_T \norm{\varphi}_{2,T}
	+ \norm{f - \nabla \cdot \mathbf{u}_h}_T h_T \norm{\varphi}_{1,T} \\
	&+ \norm{\nabla (s-\tilde{p}_h)}_T h_T \norm{\varphi}_{2,T} \le C \sum_{T \in \mathcal{T}_h} \norm{\mathbf{K}^{-1}\mathbf{u}_h + \nabla s}_T  h_T \norm{e_p}_{T} \\
	&+ \norm{f - \nabla \cdot \mathbf{u}_h}_T h_T \norm{e_p}_{T} 
	+\norm{\nabla(s-\tilde{p}_h)}_T h_T \norm{e_p}_T  \\
	&\le C \sum_{T \in \mathcal{T}_h} \left(  \norm{\mathbf{K}^{-1}\mathbf{u}_h + \nabla s}_T  h_T \norm{e_p}_{T} 
	+\norm{f - \nabla \cdot \mathbf{u}_h}_T h_T \norm{e_p}_{T} \right)\\
	&+ C \sum_{T \in \mathcal{T}_h} \norm{\nabla(s-\tilde{p}_h)}_T h_T \norm{e_p}_T 
	\end{align*}
	We apply the Young inequality and (\ref{eqn:ellipreg}) inequality complete the proof.
\end{proof}
\begin{remark} 
	We note that in case of assumption on $\tilde{p}_h$ weakly continuous between elements and across interface, the jump terms would vanish.
\end{remark}
\begin{remark}
	We note reconstruction of velocity can be useful in defining a \textit{posteriori} error. There is recently proposed the improvement of velocity at interface in \cite{amanbek2019recovery} which is a good candidate for evaluation of velocity error.
\end{remark}
\subsubsection{Estimates for Velocity without Saturation Assumption}

\begin{thm}
	Let $s \in H^1(\Omega_i)$. There exists a constant $C$ independent of $h$ such that
	\begin{align*}
	\norm{e_{\mathbf{u}}}^2 \le C \sum_{T\in \mathcal{T}_h} \Big\{ \eta_P + \eta_{R} \Big\}
	\end{align*}
	where, for all $T \in \mathcal{T}_h $
	\begin{align*}
	\eta_P &= \norm{\mathbf{K}^{-1}\mathbf{u}_h + \nabla s}^2_T  \\
	\eta_{R} &= \norm{f - \nabla \cdot \mathbf{u}_h}^2_T h^2_T 
	\end{align*}
\end{thm}

\begin{proof}
	
	For velocity we can similarly consider the following set of equations
	
	\begin{align*} 
	\norm{\mathbf{K}^{-\frac{1}{2}}e_{\mathbf{u}}}^2  &= \left( \mathbf{K}^{-1}\left(\mathbf{u} - \mathbf{u}_h \right), \mathbf{u} - \mathbf{u}_h  \right) = \left( \mathbf{K}^{-1} \mathbf{u}, \mathbf{u} - \mathbf{u}_h  \right) - \left( \mathbf{K}^{-1} \mathbf{u}_h, \mathbf{u} - \mathbf{u}_h  \right) = \\
	&= \left( p, \nabla \cdot \left( \mathbf{u} - \mathbf{u}_h \right)  \right) - \left( \mathbf{K}^{-1} \mathbf{u}_h, \mathbf{u} - \mathbf{u}_h  \right) \\ &= \left( p - s, \nabla \cdot \left( \mathbf{u} - \mathbf{u}_h \right)  \right)+\left( s, \nabla \cdot \left( \mathbf{u} - \mathbf{u}_h \right)  \right) - \left( \mathbf{K}^{-1} \mathbf{u}_h, \mathbf{u} - \mathbf{u}_h  \right) \\
	&=\underbrace{ \left( p - s, f - \nabla \cdot \mathbf{u}_h  \right)}_{\mathbb{T}_1} + \underbrace{ \sum_{T \in \mathcal{T}_h} \{\left( s, \nabla \cdot \left( \mathbf{u} - \mathbf{u}_h \right)  \right) - \left( \mathbf{K}^{-1} \mathbf{u}_h, \mathbf{u} - \mathbf{u}_h  \right) \} }_{\mathbb{T}_2}
	\end{align*}
	
	\begin{align*}
	\mathbb{T}_1 &\le |\left( p - s, \nabla \cdot \left( \mathbf{u} - \mathbf{u}_h \right)  \right)| = |\left( p - s - P_h(p-s), f - \nabla \cdot \mathbf{u}_h\right)| \\
	&\le \sum_{T \in \mathcal{T}_h} \norm{f-\nabla \cdot \mathbf{u}_h} \norm{p-s-P_h(p-s)} \\
	&\le \sum_{T \in \mathcal{T}_h}  \norm{f-\nabla \cdot \mathbf{u}_h} C_{P,T} h_T \norm{\nabla (p-s)} \\
	&\le \sum_{T \in \mathcal{T}_h} \norm{f-\nabla \cdot \mathbf{u}_h} C_{P,T} h_T \norm{-\mathbf{K}^{-1} \mathbf{u} + \mathbf{K}^{-1} \mathbf{u}_h - \mathbf{K}^{-1} \mathbf{u}_h - \nabla s} \\
	&\le \frac{5}{2} \sum_{T \in \mathcal{T}_h} C^2_{P,T}  h^2_T \norm{f-\nabla \cdot \mathbf{u}_h}^2 + \frac{1}{6} \norm{\mathbf{K}^{-1} \left( \mathbf{u} - \mathbf{u}_h \right) }^2 + \frac{1}{4}\norm{ \mathbf{K}^{-1} \mathbf{u}_h + \nabla s}^2 
	\end{align*}
	$k_{min}$ can be included as $\varepsilon$ term in the Young's inequality which helps to cancel $\mathbf{K}$ terms values.
	
	
	\begin{align*}
	\mathbb{T}_2 &= \sum_{T \in \mathcal{T}_h} \{\left( s, \nabla \cdot \left( \mathbf{u} - \mathbf{u}_h \right)  \right) - \left( \mathbf{K}^{-1} \mathbf{u}_h, \mathbf{u} - \mathbf{u}_h  \right) \} \\
	&= \sum_{T \in \mathcal{T}_h} \{-\left( \nabla s, \mathbf{u} - \mathbf{u}_h \right)_{T} + \langle s, \left( \mathbf{u} - \mathbf{u}_h\right) \cdot \nu \rangle_{\partial T} - \left( \mathbf{K}^{-1} \mathbf{u}_h, \mathbf{u} - \mathbf{u}_h  \right)_{T} \}  \\
	&= - \sum_{T \in \mathcal{T}_h} \{ \left( \mathbf{K}^{-1} \mathbf{u}_h + \nabla s, \mathbf{u} - \mathbf{u}_h \right)_{T}  +  \langle s, \left( \mathbf{u} - \mathbf{u}_h\right) \cdot \nu \rangle_{\partial T} \} \\
	&= - \sum_{T \in \mathcal{T}_h}  \left( \mathbf{K}^{-1} \mathbf{u}_h + \nabla s, \mathbf{u} - \mathbf{u}_h \right)_{T}  + \sum_{T \in \mathcal{T}_h(\Omega^*)} \langle s, \left( \mathbf{u} - \mathbf{u}_h\right) \cdot \nu \rangle_{\partial T} 
	\end{align*}
	
	We consider a last term above expression and use the trace inequality. We note $s$ is smooth in $\Omega_i$, jumps only at the interface $\Gamma$. 	
	
	We recall that
	\begin{align*}
	\norm{\left( \mathbf{u} - \mathbf{u}_h\right) \cdot \nu }^2_{-1/2, \partial T} \le C \left( \norm{\mathbf{u} - \mathbf{u}_h}^2_{T}+h^2_T \norm{\nabla \cdot \left( \mathbf{u} - \mathbf{u}_h\right)}^2\right)
	\end{align*}
	\begin{align*}
	\norm{ s-p }^2_{1/2, \partial T} \le C_1 \left( \norm{\nabla (s-p)}^2_{T}+h^2_T \norm{\left(s-p \right)^2 }\right) \le C \norm{\nabla (s-p)}^2_{T}
	\end{align*}
	by the discrete Poincare inequality or Cauchy-Schwarz inequality.
	
	\begin{align*}
	&\sum_{T \in \mathcal{T}_h, \mathcal{T}_h \cap \Gamma \ne 0 } \langle s, \left( \mathbf{u} - \mathbf{u}_h\right) \cdot \nu \rangle_{\partial T} =      \sum_{T \in \mathcal{T}_h(\Omega^*)} \langle s-p, \left( \mathbf{u} - \mathbf{u}_h\right) \cdot \nu \rangle_{\partial T} \\
	&\le \sum_{T \in \mathcal{T}_h(\Omega^*)}  \norm{s-p}_{1/2, \partial T} \norm{\left( \mathbf{u} - \mathbf{u}_h\right) \cdot \nu }_{-1/2, \partial T} \\
	&\le \Big\{ \sum_{T \in \mathcal{T}_h(\Omega^*)}  \norm{s-p}^2_{1/2, \partial T}    \Big\}^{1/2} \Big\{\sum_{T \in \mathcal{T}_h} \norm{\left( \mathbf{u} - \mathbf{u}_h\right) \cdot \nu }^2_{-1/2, \partial T} \Big\}^{1/2} \\
	& \le C \Big\{ \sum_{T \in \mathcal{T}_h(\Omega^*)}  \norm{\nabla (s-p)}^2_{T}   \Big\}^{1/2} \Big\{\sum_{T \in \mathcal{T}_h(\Omega^*)} \norm{\mathbf{u} - \mathbf{u}_h}^2_{T}+h^2_T \norm{\nabla \cdot \left( \mathbf{u} - \mathbf{u}_h\right)}^2 \Big\}^{1/2}  \\
	& \le C \Big\{ \sum_{T \in \mathcal{T}_h(\Omega^*)}  \norm{\nabla (s-p)}^2_{T}   \Big\}^{1/2} \Big\{\sum_{T \in \mathcal{T}_h(\Omega^*)} \norm{\mathbf{u} - \mathbf{u}_h}^2_{T}+h^2_T \norm{ f - \nabla \cdot \mathbf{u}_h}^2_{T} \Big\}^{1/2} \\
	& \le C^2  \sum_{T \in \mathcal{T}_h(\Omega^*)}  \norm{\nabla (s-p)}^2_{T} +\frac{1}{4} \norm{\mathbf{u} - \mathbf{u}_h}_T^2+\frac{1}{4}h^2\norm{ f - \nabla \cdot \mathbf{u}_h}_T^2
	\end{align*}

	\begin{align*}
	C \sum_{T \in \mathcal{T}_h(\Omega^*)} \norm{\nabla (p-s)}^2_{T} &\le C  \sum_{T \in \mathcal{T}_h(\Omega^*)} \norm{-\mathbf{K}^{-1} \mathbf{u} + \mathbf{K}^{-1} \mathbf{u}_h - \mathbf{K}^{-1} \mathbf{u}_h - \nabla s}_T \\
	&\le C  \sum_{T \in \mathcal{T}_h(\Omega^*)} \norm{-\mathbf{K}^{-1} \mathbf{u} + \mathbf{K}^{-1} \mathbf{u}_h}_T \norm{\mathbf{K}^{-1} \mathbf{u}_h + \nabla s}_T \\
	&\le  \sum_{T \in \mathcal{T}_h(\Omega^*)} \frac{1}{4} \norm{\mathbf{K}^{-1} \left( \mathbf{u} - \mathbf{u}_h \right) }_T^2 + C^2\norm{ \mathbf{K}^{-1} \mathbf{u}_h + \nabla s}_T^2  
	\end{align*}
	Therefore, 
	\begin{align*}
	\mathbb{T}_2 \le C \sum_{T \in \mathcal{T}_h}  \norm{\left( \mathbf{u} - \mathbf{u}_h \right) }^2_T + \norm{ \mathbf{K}^{-1} \mathbf{u}_h + \nabla s}^2_T 
	\end{align*}
	
\end{proof}
\subsection{Practical construction of the post-processed pressure.} \label{section:postprocessing}
For given Enhanced Velocity finite element approximation, we briefly describe the construction procedure of $\tilde{p}_h$, $s_h$ and an illustration of implementation in the two dimensional case.
\subsubsection{Construction of $\tilde{p}_h$.}
Restricting Enhanced Velocity space, we can denote $\widehat{\mathbf{V}}_h$ be spaces omitting interface constraints $\mathbf{V}^{\Gamma}$, so $\widehat{\mathbf{V}}_{h,i} :=\bigoplus_{i=1}^n \mathbf{V}_{h,i}(T)$ and then $\widehat{\mathbf{V}}_h :=\bigoplus_{i=1}^n \widehat{\mathbf{V}}_{h,i}$.
Let $\mathbf{u}_h$, $p_h$ be the solution of equations (\ref{eq:EV_eq1}) - (\ref{eq:EV_eq2}). We first compute Lagrange multipliers for each element. We define $\lambda_{h, T} \in \Lambda_h$, which is piecewise constant polynomials at edge or face, 
\begin{equation}\label{eq:lagrangemultipliera_fluxreconstr}
\langle \lambda_{h, T}, \mathbf{v}_h \cdot \mathbf{n}_T  \rangle_e :=
\left(\mathbf{K}^{-1}\mathbf{u}_h,\rm  v_h\right)_{T}-\left(p_h, \nabla \cdot \mathbf{v}_h \right)_{T} \qquad \forall \mathbf{v}_h \in  \widehat{\mathbf{V}}_h\left( T \right)
\end{equation}
where the element $T \in \mathcal{T}_{h} $ and its side $e$. We note that basis functions are same for $\mathbf{V}^{\Gamma}$ and $\widehat{\mathbf{V}}_h$. We employ the $L^2$ projected velocity from the interface, which has a finer enhanced velocity approximation, to the edge or face of subdomain element and the formulation is provided in the next subsection. We denote polynomial space $\widetilde{W}_h$ in the following manner
\begin{align}
\widetilde{W}_h=\{\varphi_h : \langle \llbracket  \varphi_h \rrbracket, \psi_h \rangle_e=0 \qquad \forall e \in \mathcal{E}^{int}_h \cup \mathcal{E}^{ext}_h  , \forall \psi_h \in \mathbb{Q}_m(e) \}
\end{align}
where $\mathbb{Q}_m$ is standard notation of space that is defined in \cite{arbogast1995implementation, pencheva2013robust}.
We next set the post-processed $\tilde{p}_h$ which is proposed in \cite{arbogast1995implementation} and the construction is performed with the following properties, for each $T \in \mathcal{T}_{h} $

\begin{align} 
(\tilde{p}_h, w_h)_T&=(p_h, w_h)_T \qquad \forall w_h \in \widetilde{W}_h(T), \label{eqn:4_32aa} \\
\langle \tilde{p}_h, \mu_h\rangle_e&=\langle \lambda_h, \mu_h \rangle_e \qquad \forall \mu_h \in \Lambda_h(e), \forall e \in \partial T \label{eqn:4_33aa}.	
\end{align}

\subsubsection{Construction of $s_h$.}

We propose to construct the $s_h$ in each subdomain $\Omega_i$ that has the conforming mesh in order to be an efficient in computation. Construction of $s_h$  involves the averaging operator $\mathcal{I}_{\rm av} :\mathbb{Q}_k(\mathcal{T}_h) \rightarrow \mathbb{Q}_k(\mathcal{T}_h) \cap H_0^1(\Omega_i) $. For definition of $\mathbb{Q}_m$ we refer reader to \cite{arbogast1995implementation, pencheva2013robust}. The operator is called Oswald operator and appeared in \cite{pencheva2013robust, ainsworth2005robust, karakashian2003posteriori, vohralik2010unified} and the analysis can be found in \cite{burman2007continuous, karakashian2003posteriori}. It is interesting to note that the mapping of the gradient of pressure through Oswald operator also considered in \cite{zienkiewicz1987simple}.
For given $\varphi_h \in \mathbb{Q}_m(\mathcal{T}_{h})$, we regard  the values of $\mathcal{I}_{\rm av}(\varphi_h)$ as being defined at a Lagrange node $V \in \Omega$ by averaging $\varphi_h$ values associated this node,  
\begin{align} \label{eq:oswaldavg}
\mathcal{I}_{\rm av}(\varphi_h)(V)=\frac{1}{\vert \mathcal{T}_{h} \vert} \sum_{T \in \mathcal{T}_{h}} \varphi_h \vert_T (V)
\end{align}
where $\vert  A \vert$ is cardinality of sets $A$ and $\mathcal{T}_{h} $ is all collection of $T \in \mathcal{T}_h$ for fixed $V$. One can see that $\mathcal{T}_{h}(V)=\varphi(V)$ at those nodes that are inside of given $T \in \mathcal{T}_{h}$. We set the value of $\mathcal{I}_{\rm av}(\varphi_h)$ is zero at boundary nodes. For the convinience of the reader we restate the relevant lemma from \cite{pencheva2013robust, burman2007continuous, karakashian2003posteriori} that is an important property of such construction. 
\begin{lem}
	Let $\mathcal{T}_h$ be shape regular, let $\varphi_h \in \mathbb{Q}_m(\mathcal{T}_h)$ and let $\mathcal{I}_{\rm av}(\varphi_h)$ be constructed as specified above. Then
	\begin{align}
	\norm{\nabla (\varphi_h- I_{\textrm{av}}(\varphi_h))}^2_{T} \le C \sum_{e \in \hat{\mathcal{E}}_h} h^{-\frac{1}{2}}_e \norm{ \llbracket \varphi_h \rrbracket}^2_e   \label{eq:s_tildepineqaa}
	\end{align}
\end{lem}
\noindent for all $T \in \mathcal{T}_h$ and $C$ depends only on the space dimension $d$, on the maximal polynomial degree $n$, and on the shape regularity parameter $\kappa_{\mathcal{T}}$.
Now in our setting we define recovered pressure $s_h$ for the locally post-processed $\tilde{p}_h$ as follows.
\begin{equation*}
s_h:=\mathcal{I}_{\rm av}(\tilde{p}_h)
\end{equation*}


\subsubsection{Implementation steps of construction}
We provide a brief steps of numerical implementation of post-processed pressure in two dimensional case. Based on piecewise pressure and velocity from the lowest order Raviart-Thomas spaces over rectangles our aim to reconstruct smoother pressure $s_h$.
For given element $T \in \mathcal{T}_h(\Omega_i)$, the main steps are
\begin{enumerate}
	\item Evaluate $\lambda_{h, T}$ at edge $e_j$, $j=1,..4$ based on $(\mathbf{u}_h, p_h)$,
	\item Compute $\tilde{p}_h$ from known $\lambda_{h, T}$, and  $p_h$ by using equation (\ref{eq:lagrangemultipliera_fluxreconstr}),
	\item Based on $\tilde{p}_h$ compute $s_h$ equation (\ref{eq:oswaldavg}) at Lagrange nodes in $\Omega_i$.
\end{enumerate}

Step 1 is standard computation of Lagrange multiplier for each element. In step 2, we are relying on higher order polynomial, in our case, it is Span$\{1, x, y, x^2, y^2\}$. It is sufficient to store coefficients of polynomials in the code. In step 3, we use Span$\{1, x, y, x^2, y^2, xy, x^2y, xy^2, x^2y^2\}$ and 9 Lagrange nodes of rectangle elements that are four rectangle nodes, four midpoints at edge and center of rectangle. This case each node requires to find neighboring elements values to compute coefficients of $s_h$.

\section{Numerical Examples}
We conduct several numerical experiments to show a \textit{posteriori} error bounds for two-dimensional flow problems. We set same domain $\Omega = \left(0, 1\right) \times \left(0, 1\right)$ and $H/h =2$, where $H$ is coarse subdomain discretization size and $h$ is fine subdomain discretization size, for all examples. Initial subdomains grids $\mathcal{T}_h$ are chosen in way that has a checkerboard pattern for subdomains. Examples of such discretization are shown in Figure \ref{fig:posteriorisubdomainmesh}. We focus on the flux error estimators which is key in flow and transport modeling and the comparision of actual error $\vertiii{\mathbf{u}-\mathbf{u}_h}_{*}$, which is defined as follows:
$ \vertiii {\mathbf{v}}^2_*:=\norm{\mathbf{K}^{-\frac{1}{2}} \mathbf{v}}^2, \; \mathbf{v} \in L^2(\Omega).$
The localization of the element level computation brings efficiency in the many domains setting for EV scheme methodology. 
In the first numerical example, the residual-based error estimators and implicit error estimators are compared for pressure and velocity errors. Second example shows the post-processed error estimators (implicit) for heterogeneous porous media. In the third example, we demonstrate the advantages of estimator, $\eta_{P}$, as indicator for three by three subdomains.


\begin{figure}[H]
	\centering
	\includegraphics[width=0.45\linewidth]{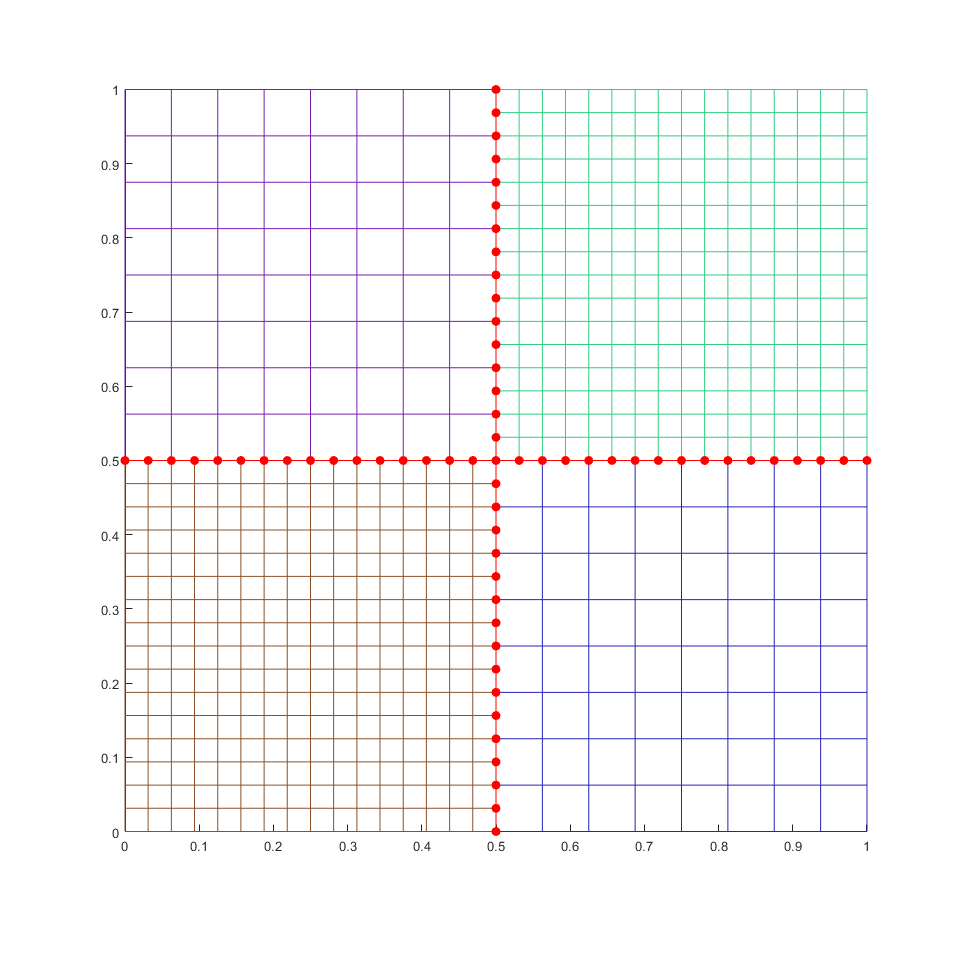}
	\includegraphics[width=0.45\linewidth]{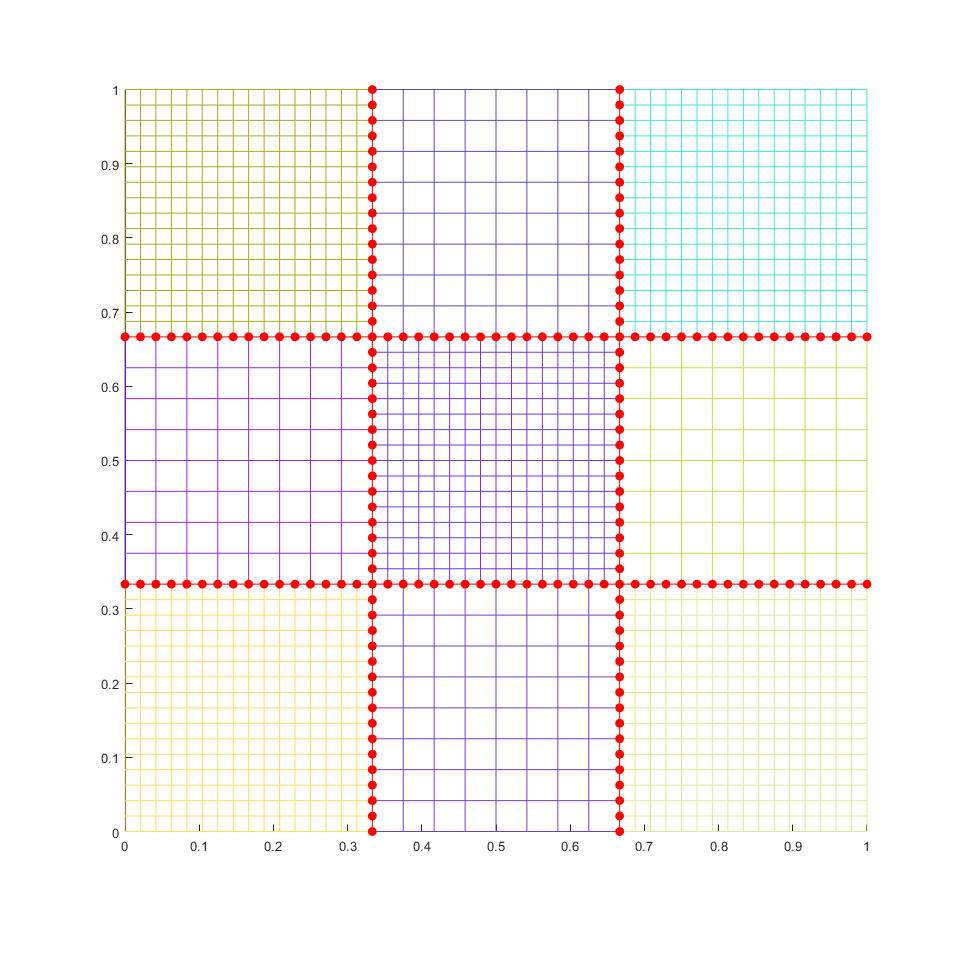}
	\caption{Example of non-matching grids for subdomains.}
	\label{fig:posteriorisubdomainmesh}
\end{figure}

\subsection*{Numerical Example 1}
First example tests for uniform permeability, so  $\mathbf{K} = \mathbf{I}$. We report the velocity error and corresponding a posteriori error estimators. We compute the source term and boundary conditions according to the analytical solution, which is taken as follows
\begin{align*}
p(x,y) = 1000 xy e^{-10(x^2+y^2)}.
\end{align*}

We want to illustrate in this example of the pressure estimator. However, the importance of flux is key in the mesh refinement, since the flow coupled with transport by the flux.

\begin{figure}[H]
	\begin{center}
		\includegraphics[width=0.32\linewidth,trim=1.5cm 0cm 1.4cm 0cm, clip]{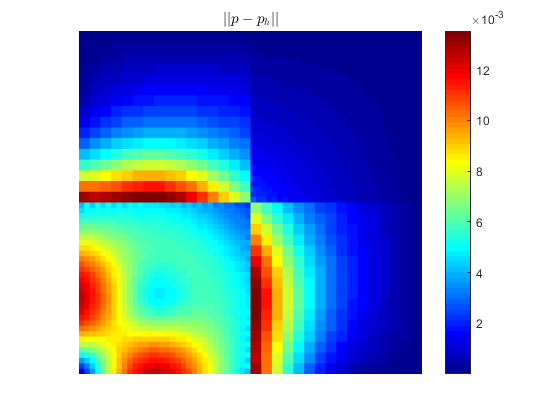}
		\hfill
		\includegraphics[width=0.32\linewidth,trim=1.5cm 0cm 1.4cm 0cm, clip]{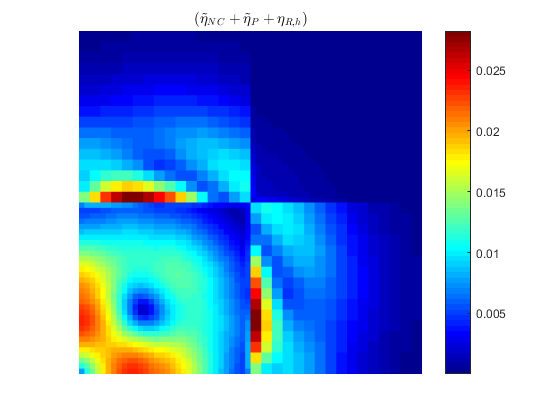}
		\hfill
		\includegraphics[width=0.32\linewidth,trim=1.5cm 0cm 1.4cm 0cm, clip]{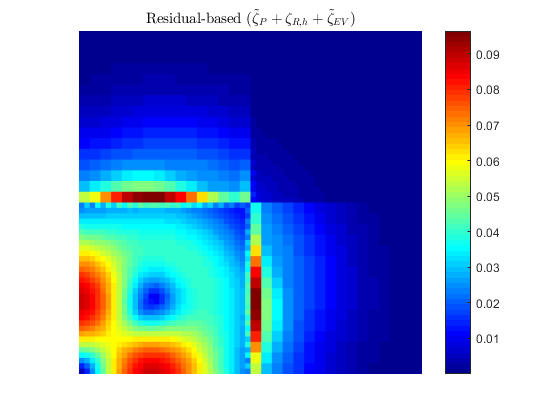}
		\caption{Actual(left), estimated $\eta$ (center) and the residual-based $\zeta $ (right)  pressure error distribution on a uniformly refined mesh.}
		\label{fig:preserror_N10}
	\end{center}
\end{figure}

\begin{figure}[H]
	\begin{center}
		\includegraphics[width=0.32\linewidth,trim=1.5cm 0cm 1.5cm 0cm, clip]{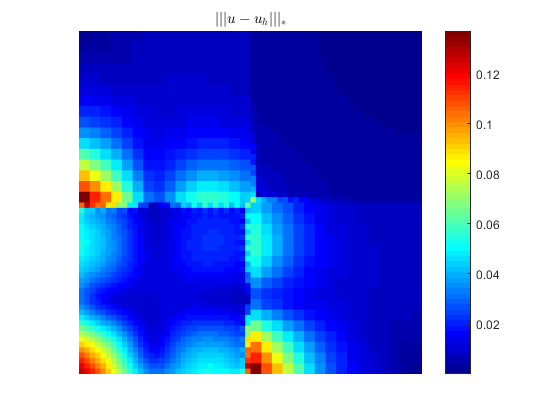}
		\hfill
		\includegraphics[width=0.32\linewidth,trim=1.5cm 0cm 1.5cm 0cm, clip]{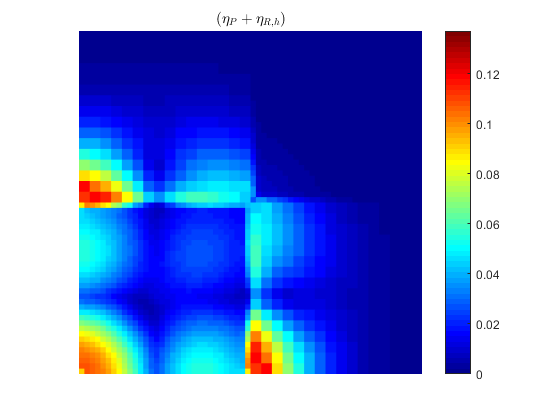}
		\hfill
		\includegraphics[width=0.32\linewidth,trim=1.5cm 0cm 1.5cm 0cm, clip]{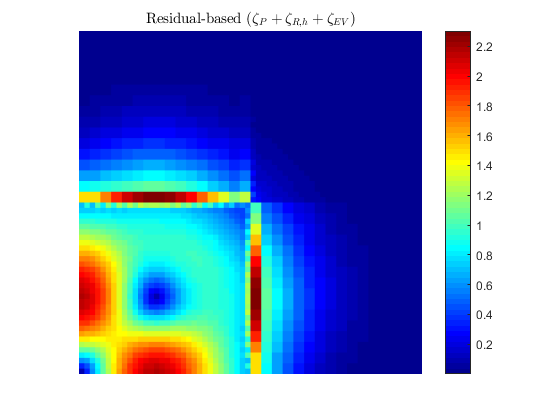}
		\caption{Actual(left), estimated $\eta$ (center) and the residual-based $\zeta $ (right) velocity error distribution on a uniformly refined mesh.}
		\label{fig:fluxerror_N10}
	\end{center}
\end{figure}

As can be seen from Figure \ref{fig:preserror_N10}, the pressure error was detected well by $\eta$ and $\zeta$, however, error estimator with postprocessing is closer to the actual error. We can see from Figure \ref{fig:fluxerror_N10} the residual-based estimators of velocity error are not good in detection of velocity error, on the other hand, the estimator with postprocessing captures the actual error very well. To be specific, not only error within subdomains but also error at the interface are detected. The result indicates that the fine-grid subdomains error is also predicted for further refinement process, which might depends on chosen the marking strategy in adaptive setting. In this case, we note that the interface grids equals to the fine-grid subdomains mesh size and we set the coarser grid to show increased interface error. The residual-based estimators of velocity error are not good in detection of velocity error, see Figure \ref{fig:fluxerror_N10}. 

\begin{figure}[H]
	\begin{center}
		\includegraphics[width=0.6\linewidth,trim=0cm 0cm 0cm 0cm, clip]{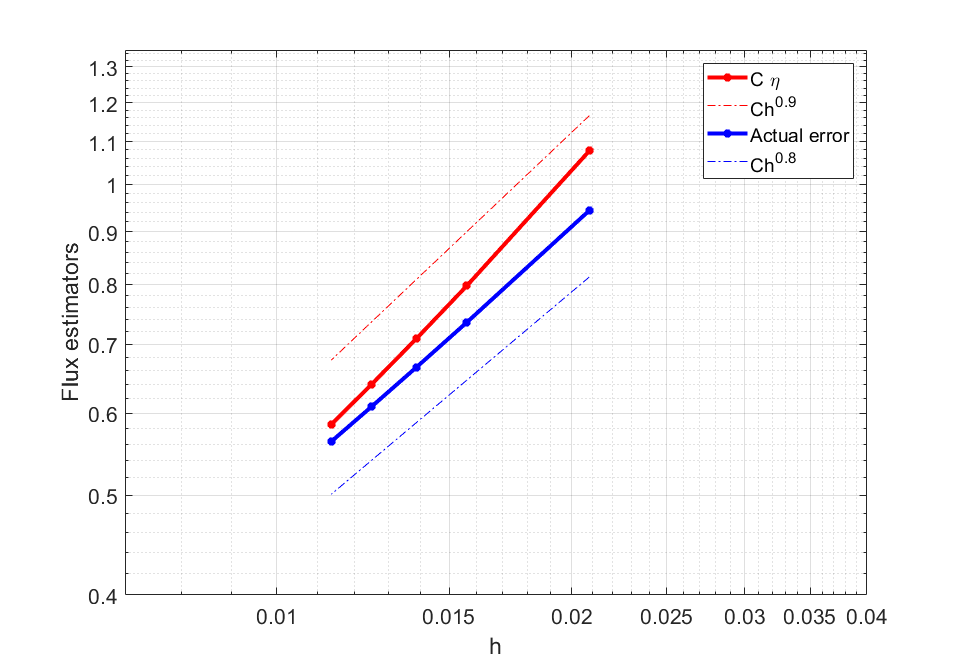}
		\caption{Computed estimates and actual flux error distribution on a uniformly refined mesh.}
		\label{fig:fluxglobalerror_N10}
	\end{center}
\end{figure}

From now we will refer as error estimator to the post-processed error estimators. As can be seen from Figure \ref{fig:fluxglobalerror_N10}, we compare convergence rate of actual and estimated flux errors against to the coarse subdomain mesh size. The coarse subdomain discretization are $h=\{1/48, 1/64, 1/72, 1/80, 1/88\}$ and the fine subdomain discretization is two times smaller, i.e. \\
$h_f=\{1/96, 1/128, 1/144, 1/160, 1/172\}$. The convergence rate are almost the same and it implies that the effectivity indexes, resulting from the ratios of the estimate over the error, is decreasing slowly to constant number. By increasing number of degrees the estimator approaches the actual error at acceptable precision level. 
\begin{figure}[H]
	\begin{center}
		\includegraphics[width=0.6\linewidth,trim=0cm 0cm 0cm 0cm, clip]{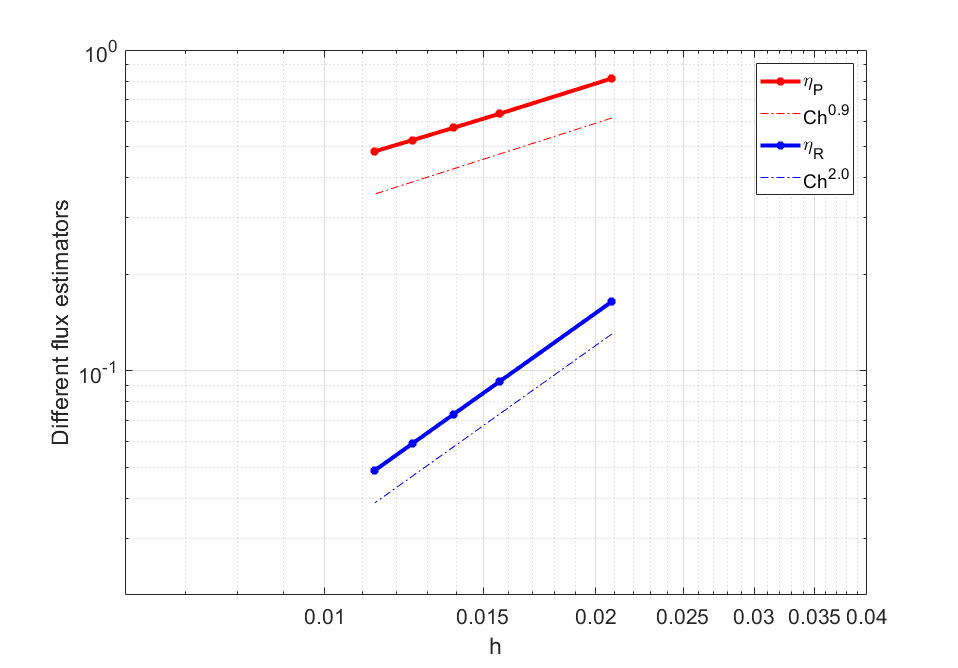}
		\caption{Different flux estimators distribution on a uniformly refined mesh.}
		\label{fig:fluxdifferentestimators_N10}
	\end{center}
\end{figure}
We have also considered different components of flux estimator, i.e. $\eta_P$ and $\eta_{R}$. For visual representation of the convergence rates the reader refers to Figure \ref{fig:fluxdifferentestimators_N10}. From this comparison we can see that the estimator $\eta_P$ and $\vertiii{\mathbf{u}-\mathbf{u}_h}_*$ converges similarly as $\mathcal{O}(h^{0.9})$ and $\mathcal{O}(h^{0.8})$, respectively. It is interesting to notice that $\eta_P$ is a good error indicator, since it is similar to the entire-domain error distribution. On the other hand, the estimator $\eta_{R_h}$ converges as $\mathcal{O}(h^2)$, because $f$ is smooth. We expected that the estimator at interface $\eta_{EV}$ converges  here $\mathcal{O}(h^{0.5})$ which is slower than $\mathcal{O}(h^{0.8})$.

\subsection*{Numerical Example 2}
We consider the a diagonal heterogeneous permeability profiles of the medium
\begin{align*}
\mathbf{K} =
\begin{bmatrix} 
e^{\cos(4 \pi x)\cos(2 \pi y)+3\sin(5 \pi x) \cos(3 \pi y)}    &   0 \\ 
0 &    e^{\cos(4 \pi x)\cos(2 \pi y)+3\sin(5 \pi x) \cos(3 \pi y)}
\end{bmatrix} .
\end{align*}

We compute the source term and boundary conditions according to the analytical solution, which is taken as follows
\begin{align*}
p(x,y) = \sin(\pi x) \sin(\pi y).
\end{align*}

\begin{figure}[H]
	\begin{center}
		\includegraphics[width=0.47\linewidth,trim=1.5cm 0cm 1.5cm 0cm, clip]{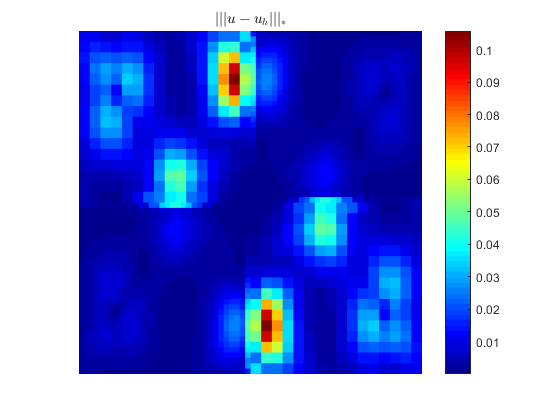}
		\hfill
		\includegraphics[width=0.47\linewidth,trim=1.5cm 0cm 1.5cm 0cm, clip]{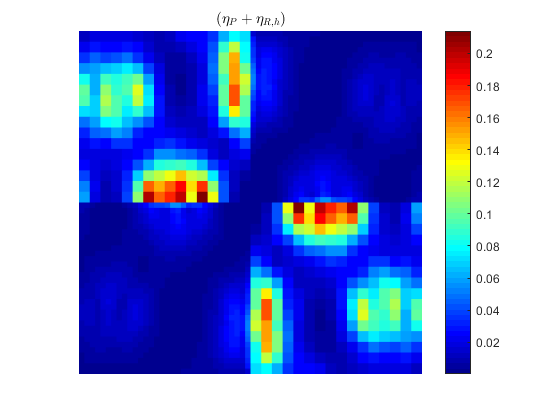}
		\caption{Actual(left) and estimated (right) flux error distribution on a uniformly refined mesh.}
		\label{fig:fluxerrorN4}
	\end{center}
\end{figure}

As shown in Figure \ref{fig:fluxerrorN4}, the actual spatial distribution of the flux error is predicted by the estimator well. The highly heterogeneous porous medium leads to increase of the estimate values. Some part is over predicted since the residual term has large values on high permeability regions compare the potential estimator term.

\begin{figure}[H]
	\begin{center}
		\includegraphics[width=0.6\linewidth,trim=0cm 0cm 0cm 0cm, clip]{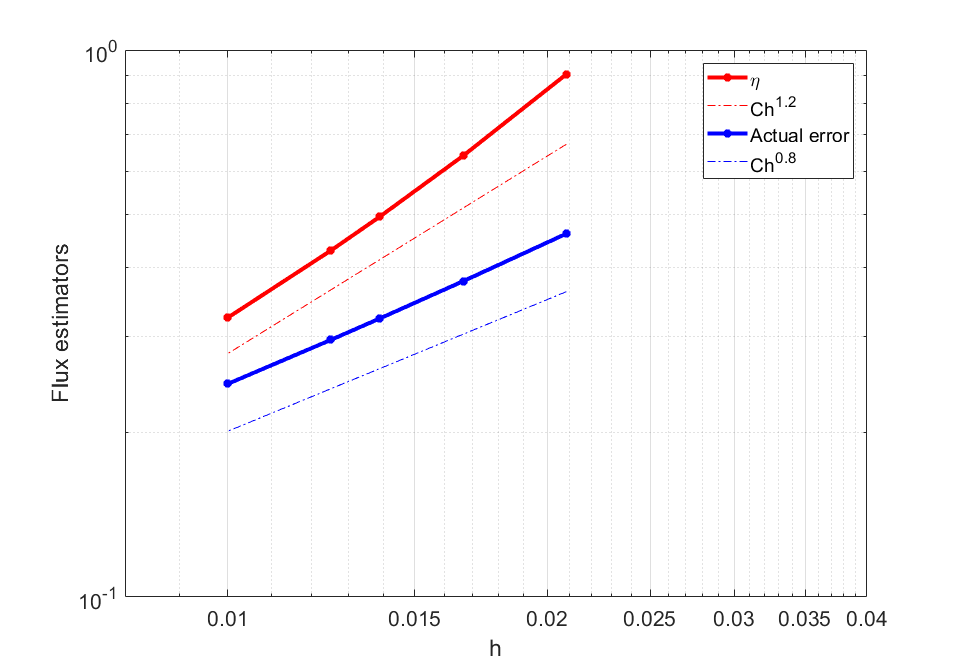}
		\caption{Computed estimates and actual flux error distribution on a uniformly refined mesh.}
		\label{fig:fluxglobalerror_N4}
	\end{center}
\end{figure}

As can be seen from Figure \ref{fig:fluxglobalerror_N4}, the comparison of actual and estimated flux error against to the coarse subdomain mesh size. The estimate is upper bound on the error as predicted by theory. The coarse subdomain discretization are $h=\{1/36, 1/48, 1/64, 1/80, 1/100\}$ and the fine subdomain discretization is two times smaller, i.e. $h=\{1/72, 1/96, 1/128, 1/160, 1/200\}$. The convergence rate are the same and it implies that the effectivity indices, resulting from the ratios of the estimate over the error, is decreasing quickly to constant number.

\begin{figure}[H]
	\begin{center}
		\includegraphics[width=0.6\linewidth,trim=0cm 0cm 0cm 0cm, clip]{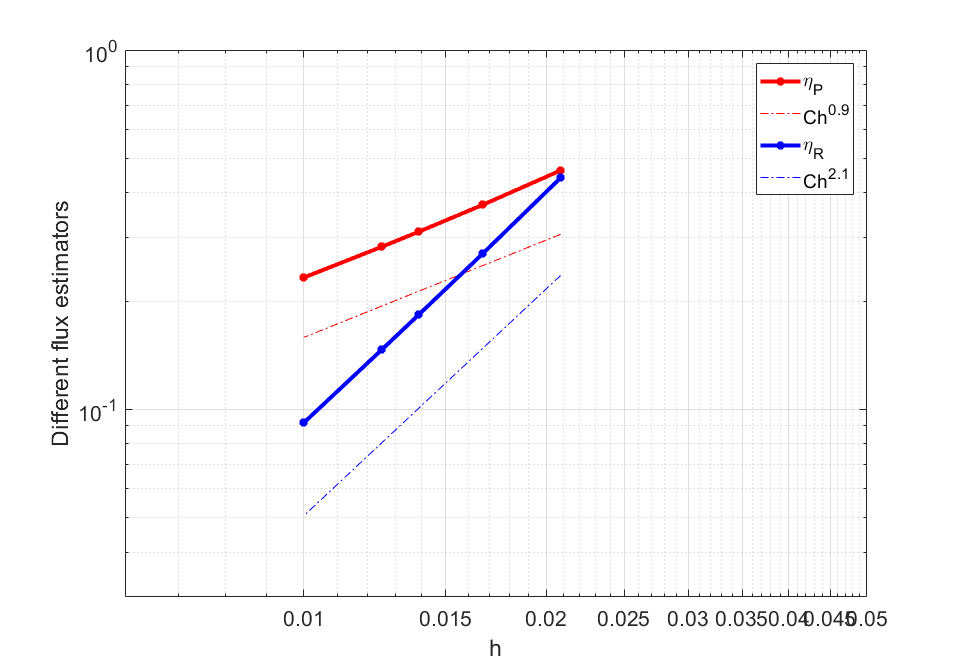}
		\caption{Different flux estimators distribution on a uniformly refined mesh.}
		\label{fig:fluxdifferentestimators_N4}
	\end{center}
\end{figure}

We have also considered different flux estimator components $\eta_P$, $\eta_{R_h}$ and $\eta_{EV}$. For visual representation of the convergence rate the reader refers to Figure \ref{fig:fluxdifferentestimators_N4}. From this comparison we can see that the estimator $\eta_P$ and $\vertiii{\mathbf{u}-\mathbf{u}_h}_*$ converges as similar rate $\mathcal{O}(h^{0.9})$ and $\mathcal{O}(h^{0.8})$, respectively. It is interesting to notice that $\eta_P$ is a good error indicator if you look at the entire-domain error distribution. On the other hand, the estimator $\eta_{R_h}$ converges as $\mathcal{O}(h^{2.2})$ that is much better, because $f$ is smooth.

\subsection*{Numerical Example 3}
In this numerical example, we want to show indicator advantages in the detecting flux error for three by three subdomains, see Figure \ref{fig:posteriorisubdomainmesh}. It is common to see in the subsurface flow problems $\eta_{R}$ term that might have impact near the well location due to source term. The rest of domain $\eta_{P}$ can be considered as indicator of the large error element. Therefore, we suggest to use $\eta_{P}$ as an estimator in the flow and transport problems to detect potential elements for refinement. Advantage of this estimator is a local evaluation of post-processed pressure approximation.  


We consider the absolute permeability as $\mathbf{K} = \mathbf{I}$. The solution of the flow problem is 
\begin{equation*}
p(x, y) = x(x-1)y(y-1)
\end{equation*}
and then corresponding boundary conditions and force term were computed. For visual comparison of error indicator and actual error and the $\eta_P$ the reader referred to Figure \ref{fig:fluxerrorN3}. As we discussed before, the numerical example confirms that the estimator $\eta_{P}$ is a good candidate as indicator in the adaptivity strategy.
\begin{figure}[H]
	\begin{center}
		\includegraphics[width=0.32\linewidth,trim=1.5cm 0cm 1.5cm 0cm, clip]{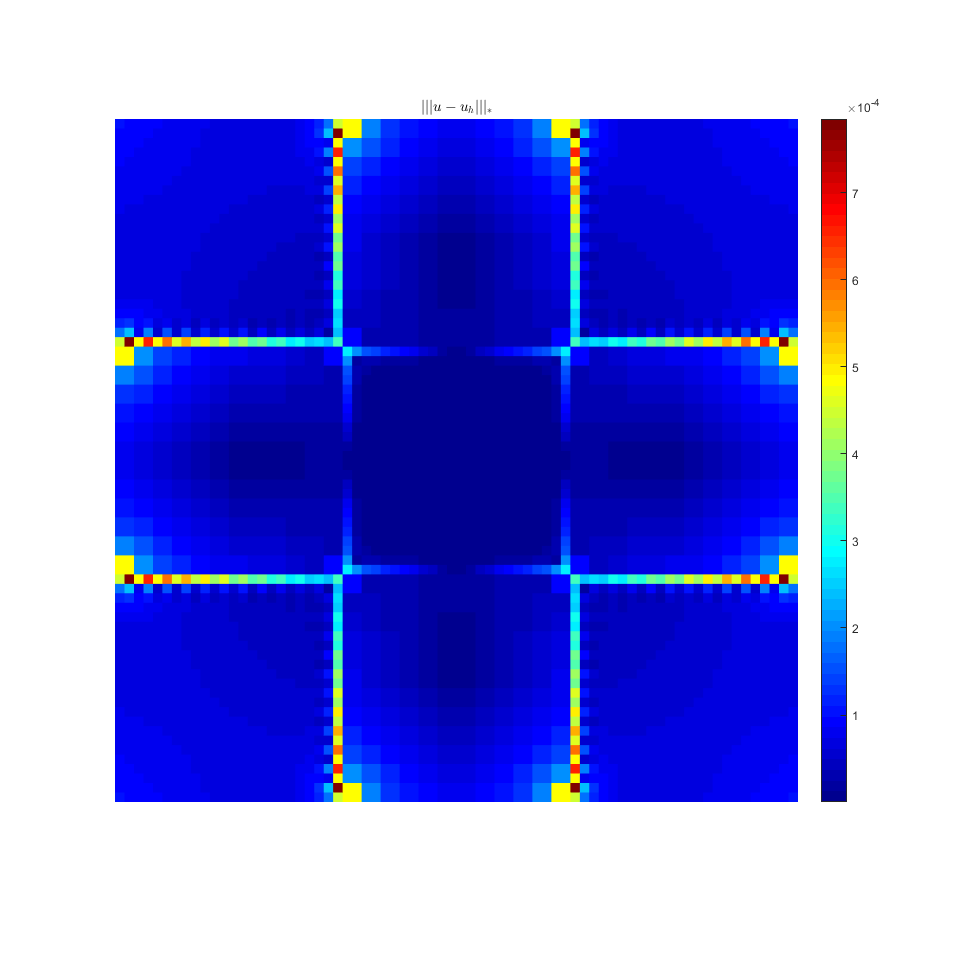}
		\hfill
		\includegraphics[width=0.32\linewidth,trim=1.5cm 0cm 1.5cm 0cm, clip]{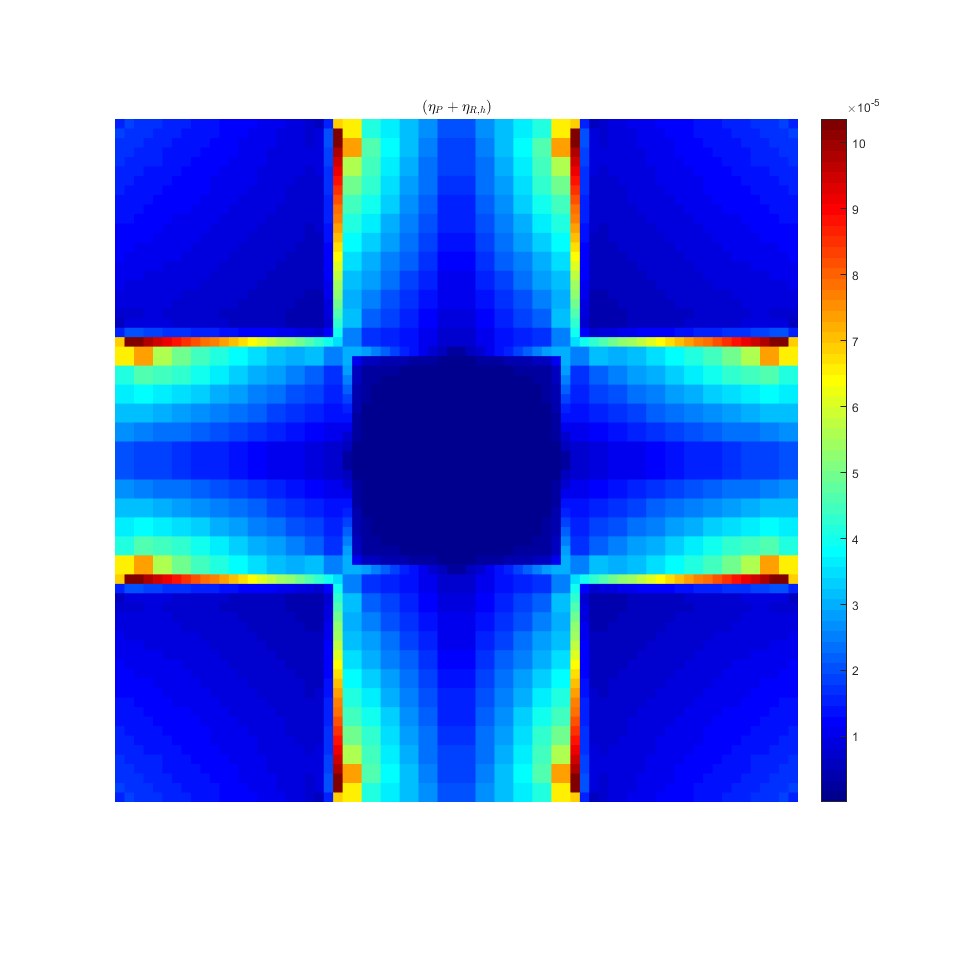}
		\hfill
		\includegraphics[width=0.32\linewidth,trim=1.5cm 0cm 1.5cm 0cm, clip]{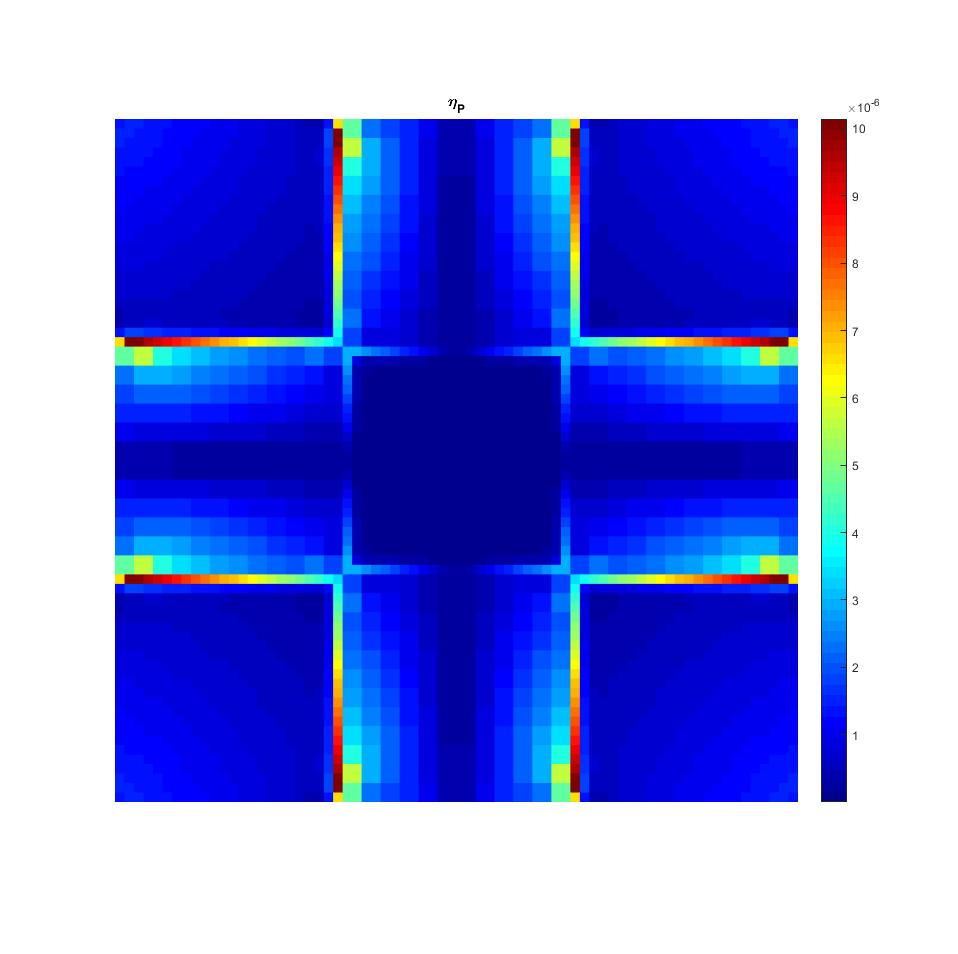}
		\caption{Actual(left), estimated (center) flux error and (right) $\eta_P$ error distribution on a uniformly refined mesh.}
		\label{fig:fluxerrorN3}
	\end{center}
\end{figure}

\section{Conclusions}
The main goal of the present research was to determine a \textit{posteriori} error estimators of Enhanced Velocity discretization for the incompressible flow problems. In this paper, we derived theoretically the explicit residual-based error estimators with the saturation assumption and the implicit error estimators, which is based on post-processed pressure, and confirmed them numerically. For pressure error, the residual-based error estimators indications are similar to the implicit error estimators. On the other hand, the flux error estimators were predicted well by the implicit error estimators. In addition, the numerical result indicates that the flux error can be detected by $\eta_P$. In the flow and transport setting, the flux is key component in coupling and therefore identifying regions of fine scale flux by using a \textit{posteriori} error analysis is key in adaptivity strategy. The findings suggest that the provided error estimates could also be useful for reservoir simulator in industrial applications. In our future research we intend to concentrate on a \textit{posteriori} error analysis of transient problems and the result of \textit{priori} error study has been demonstrated in \cite{amanbek2018priori}.    


\section*{Acknowledgements} 
First author would like to acknowledge support of the Faculty Development Competitive Research Grant (Gran No. 110119FD4502), Nazarbayev University.

\bibliographystyle{unsrt}  
\bibliography{posteriori_manuscript}

\end{document}